\documentclass[a4paper,reqno,10pt]{amsart}
\makeatletter
\newcommand*{\rom}[1]{\expandafter\@slowromancap\romannumeral #1@}
\makeatother
\usepackage{epsf,graphicx,epsfig}
\usepackage{amscd}
\usepackage{amsmath,latexsym,amssymb,amsthm}
\usepackage{fdsymbol}
\usepackage[nospace,noadjust]{cite}
\usepackage{textcomp}
\usepackage{setspace,cite}
\usepackage{mathrsfs,amsmath}
\usepackage{chemarrow}
\usepackage{lscape,fancyhdr,fancybox}
\usepackage{stmaryrd}
\usepackage[all,cmtip]{xy}
\usepackage{tikz-cd}
\usepackage{cancel}
\usepackage[usestackEOL]{stackengine}
\usetikzlibrary{shapes,arrows,decorations.markings}
\usepackage{graphicx}

\usepackage{color}
\usepackage{bbm, dsfont}
\usepackage{xcolor}
\usepackage{url}
\usepackage{enumerate}
\usepackage[mathscr]{euscript}

  \theoremstyle{plain}
\swapnumbers
    \newtheorem{thm}{Theorem}[section]
    \newtheorem{proposition}[thm]{Proposition}
     \newtheorem{theorem}[thm]{Theorem}
   
    \newtheorem{corollary}[thm]{Corollary}
    
    \newtheorem{subsec}[thm]{}
\theoremstyle{definition}
    \newtheorem{definition}[thm]{Definition}
        \newtheorem{remark}[thm]{Remark}
    \newtheorem{exam}[thm]{Example}

\theoremstyle{remark}

\title{}
\author{}
\date{}
\usepackage{amssymb}

\usepackage{hyperref}
\hypersetup{
colorlinks,
citecolor=blue,
filecolor=black,
linkcolor=blue,
urlcolor=black
}

\usepackage[left=25mm,right=25mm,top=25mm,bottom=25mm,paper=a4paper]{geometry}

\begin{document}



\title[]{Applications of Poisson cohomology to the inducibility problems and study of deformation maps}


\author{Apurba Das}
\address{Department of Mathematics,
Indian Institute of Technology, Kharagpur 721302, West Bengal, India.}
\email{apurbadas348@gmail.com, apurbadas348@maths.iitkgp.ac.in}

\author{Ramkrishna Mandal}
\address{Department of Mathematics, Indian Institute of Technology, Kharagpur 721302, West Bengal, India}
\email{ramkrishnamandal430@gmail.com}

\author{Anupam Sahoo}
\address{Department of Mathematics,
Indian Institute of Technology, Kharagpur 721302, West Bengal, India.}
\email{anupamsahoo23@gmail.com}

\begin{abstract}
This paper provides some applications of the Poisson cohomology groups introduced by Flato, Gerstenhaber and Voronov. Given an abelian extension of a Poisson algebra by a representation, we first investigate the inducibility of a pair of Poisson algebra automorphisms and show that the corresponding obstruction lies in the second Poisson cohomology group. Consequently, we obtain the Wells exact sequence connecting various automorphism groups and the second Poisson cohomology group. Subsequently, we also consider the inducibility for a pair of Poisson algebra derivations, obtain the obstruction and construct the corresponding Wells-type exact sequence. 

To get another application, we introduce the notion of a `deformation map' in a proto-twilled Poisson algebra. A deformation map unifies various well-known operators such as Poisson homomorphisms, Poisson derivations, crossed homomorphisms, Rota-Baxter operators of any weight, twisted Rota-Baxter operators, Reynolds operators and modified Rota-Baxter operators on Poisson algebras. We show that a deformation map $r$ induces a new Poisson algebra structure and a suitable representation of it. The corresponding Poisson cohomology is defined to be the cohomology of the deformation map $r$. Finally, we study the formal deformations of the operator $r$ in terms of the cohomology.
\end{abstract}

\maketitle

\medskip

\medskip

    

\noindent

\thispagestyle{empty}

\tableofcontents

\vspace{0.2cm}

\section{Introduction}\label{sec1}
Algebraic structures are better understood by associating invariants with them. Among others, the cohomology of an algebra (e.g. associative, Lie or any specific type) is an important invariant central in many studies of the underlying algebra. For example, the cohomology of an associative algebra $A$ (also known as the Hochschild cohomology of $A$) is closely related to abelian extensions of $A$ and plays a prominent role in controlling the formal one-parameter deformations of $A$ \cite{hoch}, \cite{gers}. Subsequently, the analogous results were obtained by Nijenhuis and Richardson for the cohomology of a Lie algebra $\mathfrak{g}$ (also known as the Chevalley-Eilenberg cohomology of $\mathfrak{g}$) \cite{nij-ric-1}, \cite{nij-ric-2}. On the other hand, an application of the cohomology of an abstract group to the inducibility problem was found by Wells \cite{wells} (see also \cite{jin-liu}, \cite{passi}). Later, his theory and results were followed in the context of Lie algebras and associative algebras \cite{bar-singh}, \cite{tan-xu}, \cite{bai-zhang}. An important connection between the Chevalley-Eilenberg cocycles of a Lie algebra and skeletal $2$-term $L_\infty$-algebras was discovered by Baez and Crans \cite{baez-crans}. Finally, in recent times, people have been very interested in the cohomology of various linear operators (e.g. homomorphisms, derivations, crossed homomorphisms, Rota-Baxter operators, etc.) defined in an algebra \cite{sheng-quasi}, \cite{das-mandal}. It turns out that any of these operators induces a new algebra structure and a suitable representation of it. The cohomology of the induced algebra with coefficients in that suitable representation is defined to be the cohomology of the prescribed operator. Hence the cohomology of an algebra is also useful in defining the cohomology of some linear operators.


\medskip

A Poisson algebra is an important algebraic structure that appears in the mathematical formulation of Hamiltonian mechanics \cite{vaisman}. These algebras are also central in the study of quantum groups \cite{drin}. Recall that a Poisson algebra is a commutative associative algebra endowed with a Lie bracket that satisfies the Leibniz rule. The algebra of observables in symplectic geometry, the space of smooth functions on the dual of a Lie algebra, vertex operator algebras and the semi-classical limit of any associative formal deformations of a commutative associative algebra inherit Poisson algebra structures \cite{vaisman}, \cite{bor}, \cite{kont}. Representations or modules over Poisson algebras are also well-studied in the literature \cite{fgv}, \cite{caressa}. In \cite{fgv} Flato, Gerstenhaber and Voronov defined the cohomology of a Poisson algebra by considering a bicomplex unifying the classical Harrison cochain complex of the underlying commutative associative algebra and the Chevalley-Eilenberg cochain complex of the underlying Lie algebra (see also \cite{bao-ye},\cite{bao-ye2}). In the same paper, the authors also showed that the cohomology of a Poisson algebra is closely related to abelian extensions and controls the simultaneous deformations of both the commutative associative product and the Lie bracket. Our aim in this paper is to provide the applications of Poisson cohomology to the inducibility problems and the cohomological study of various linear operators defined on a Poisson algebra (see below for brief discussions). Further applications of Poisson cohomology to the appropriate notions of $P_\infty$-algebras (strongly homotopy Poisson algebras) and Poisson $2$-algebras (categorification of Poisson algebras) will be discussed in a separate article.


\subsection{The inducibility problems} As mentioned earlier, the inducibility problem for automorphisms was first considered in a paper by Wells. Subsequently, the same problem was generalized for derivations \cite{tan-xu}, \cite{bai-zhang}.
In the context of Lie algebras, the inducibility problems for automorphisms and derivations can be stated as follows. First, let $0 \rightarrow V \xrightarrow{i} E \xrightarrow{p} \mathfrak{g} \rightarrow 0$ be an abelian extension of a Lie algebra $\mathfrak{g}$ by a given representation $V$ (i.e., the induced representation on $V$ coming from the abelian extension coincides with the prescribed one). (i) Let $\mathrm{Aut}_V (E)$ be the set of all Lie algebra automorphisms $\gamma \in \mathrm{Aut} (E)$ that satisfies $\gamma (V) \subset V$. Then there is a group homomorphism $\tau : \mathrm{Aut}_V (E) \rightarrow \mathrm{Aut} (V) \times \mathrm{Aut} (\mathfrak{g})$, $\tau (\gamma) = (\gamma |_V , p \gamma s)$, where $s$ is a section of the map $p$. Then the inducibility problem for automorphisms asks to find a necessary and sufficient condition under which a pair  $(\beta, \alpha) \in \mathrm{Aut} (V) \times \mathrm{Aut} (\mathfrak{g})$ of Lie algebra automorphisms lies in the image of $\tau$.
(ii) Let $\mathrm{Der}_V (E)$ be the space of all Lie algebra derivations $d \in \mathrm{Der} (E)$ that satisfies $d (V) \subset V$. As before, there is a map $\eta : \mathrm{Der}_V (E) \rightarrow \mathrm{Der} (V) \times \mathrm{Der}(\mathfrak{g})$ given by $\eta (d) = (d |_V , pds)$, where $s$ is a section of $p$. The inducibility problem for derivations then asks to find a necessary and sufficient condition under which a pair $(d_V, d_\mathfrak{g}) \in \mathrm{Der} (V) \times \mathrm{Der}(\mathfrak{g})$ of Lie algebra derivations lies in the image of $\eta$. It is easy to see that the above inducibility problems make sense in the context of other types of algebras. See \cite{tan-xu}, \cite{bai-zhang}, \cite{mishra-das}, \cite{goswami} and the references therein for these problems on different algebras. 

\medskip

Here we shall study the inducibility problems in the context of Poisson algebras. For these, we will introduce suitable Wells maps (one for Poisson automorphisms and another for Poisson derivations) in the context and show that the obstructions for both problems lie in the second Poisson cohomology group. Consequently, we obtain two Wells exact sequences, one connecting various Poisson automorphism groups and the second Poisson cohomology group, and the other connecting various Poisson derivation spaces and the second Poisson cohomology group.

\subsection{Deformation maps}
The notion of deformation maps in a twilled Lie algebra (or a matched pair of Lie algebras) was introduced by Agore and Militaru to study the classifying complements problem \cite{agore0}. It turns out that a deformation map in a twilled Lie algebra unifies Lie algebra homomorphisms, derivations, crossed homomorphisms, and (relative) Rota-Baxter operators of weight $0$, $1$ \cite{sheng-quasi}. However, there are operators similar to Rota-Baxter operators (e.g. twisted Rota-Baxter operators, Reynolds operators and modified Rota-Baxter operators) on Lie algebras that cannot be viewed as deformation maps in twilled Lie algebras. To put this list of operators in the same unified framework, Jiang, Sheng and Tang \cite{sheng-quasi} considered quasi-twilled Lie algebras and introduced two types of deformation maps (type-I and type-II). They further defined the cohomology groups of these deformation maps that unified the existing cohomology groups of all the well-known operators mentioned above.

\medskip

Our primary aim in this part is to define the cohomology of Poisson homomorphisms, Poisson derivations, crossed homomorphisms, Rota-Baxter operators of weight $0$ and $1$, twisted Rota-Baxter operators, Reynolds operators and modified Rota-Baxter operators on Poisson algebras. To unify all these operators by a single map (unlike type-I and type-II deformation maps), one needs to consider a bigger object than a quasi-twilled Poisson algebra. Namely, we will consider a proto-twilled Poisson algebra, which is a Poisson algebra whose underlying vector space has a direct sum decomposition into subspaces. A quasi-twilled Poisson algebra is a particular case in which a decomposed subspace is a Poisson subalgebra, and a twilled Poisson algebra is a case in which both the decomposed subspaces are Poisson subalgebras. We define the notion of a deformation map in a proto-twilled Poisson algebra (generalizing deformation map in a twilled Poisson algebra \cite{agore}) and show that a deformation map unifies all the above-mentioned operators on Poisson algebras. We show that a deformation map $r$ in a proto-twilled Poisson algebra induces a new Poisson algebra and a suitable representation. The corresponding Poisson cohomology is the cohomology of the deformation map $r$. In particular, we obtain the cohomologies of all the well-known operators mentioned above. Finally, we study linear and formal one-parameter deformations of the operator $r$ in terms of its cohomology.


\subsection{Organization} The paper is organized as follows. In Section \ref{sec2}, we recall some necessary background on Poisson algebras, including their cohomology. The inducibility problems for Poisson automorphisms and Poisson derivations are discussed in Sections \ref{sec3} and \ref{sec4}, respectively. We consider proto-twilled Poisson algebras and deformation maps in Section \ref{sec5}. In particular, we observe that various well-known operators on a Poisson algebra can be viewed as deformation maps in suitable proto-twilled Poisson algebras. In Section \ref{sec6}, we introduce the cohomology of a deformation map. In particular, we define the cohomologies of well-known operators on Poisson algebras. Finally, in Section \ref{sec7}, we study formal one-parameter deformations of a deformation map $r$ using the cohomology.

All vector spaces and algebras, (multi)linear maps, tensor products and wedge products are over a field {\bf k} of characteristic $0$ unless specified otherwise.


\medskip

\section{Poisson algebras and the Flato-Gerstenhaber-Voronov cohomology}\label{sec2}

In this section, we recall representations and cohomology of Poisson algebras as introduced by Flato, Gerstenhaber and Voronov \cite{fgv}. 

\begin{definition}
   A {\bf Poisson algebra} is a triple $(P, ~ \!  \cdot ~ \! , \{ ~, ~ \} )$ consisting of a commutative associative algebra $(P, ~ \! \cdot ~ \!)$ and a Lie algebra $(P, \{ ~, ~ \})$ both defined on a same vector space $P$ satisfying additionally the following Leibniz rule:
   \begin{align*}
       \{ x, y \cdot z \} = \{ x, y \} \cdot z + y \cdot \{ x, z \}, \text{ for } x, y, z \in P.
   \end{align*}
\end{definition}

Let $ (P, ~ \!  \cdot ~ \! , \{ ~, ~ \} )$ be a Poisson algebra. A {\bf Poisson subalgebra} of $(P, ~ \!  \cdot ~ \! , \{ ~, ~ \} )$ is a vector subspace $Q \subset P$ that is a subalgebra for both the commutative associative algebra $(P, ~ \!  \cdot ~ \!  )$ and the Lie algebra $(P, \{ ~, ~ \})$. 

Poisson algebras are the underlying algebraic structure of Poisson manifolds. The algebra of observables in symplectic geometry, the space of smooth functions $C^\infty (\mathfrak{g}^*)$ on the dual of a Lie algebra $\mathfrak{g}$, vertex operator algebras and the semi-classical limit of any associative deformations of a commutative algebra inherit Poisson algebra structures.

\begin{definition}
    Let $ (P, ~ \!  \cdot ~ \! , \{ ~, ~ \} )$ be a Poisson algebra. A {\bf representation} of $ (P, ~ \!  \cdot ~ \! , \{ ~, ~ \} )$ is a triple $(V, \mu, \rho)$ consisting of a vector space $V$ endowed with linear maps $\mu, \rho : P \rightarrow \mathrm{End}(V)$ such that
\begin{itemize}
    \item[(i)]  $(V, \mu)$ is a module over the commutative associative algebra $(P, ~ \! \cdot ~ \! )$, i.e., $\mu_{ x \cdot y} = \mu_x \mu_y$,
    \item[(ii)] $(V, \rho)$ is a representation of the Lie algebra $(P, \{ ~, ~ \})$, i.e., $\rho_{ \{ x, y \} } = \rho_x \rho_y - \rho_y \rho_x$,
\end{itemize}
for all $x, y \in P$, and the following compatibilities are hold:
    \begin{align*}
        \mu_{ \{ x, y \} } =~& \rho_x \circ \mu_y - \mu_y \circ \rho_x,\\
        \rho_{x \cdot y} =~& \mu_x \circ \rho_y + \mu_y \circ \rho_x.
    \end{align*}
\end{definition}

Let $ (P, ~ \!  \cdot ~ \! , \{ ~, ~ \} )$ be a Poisson algebra. Then the triple $(P, \mu_\mathrm{ad}, \rho_\mathrm{ad})$ is a representation, where $(\mu_\mathrm{ad})_x y = x \cdot y$ and $(\rho_\mathrm{ad})_x y = \{x, y \}$, for all $x, y \in P$. This is called the {\em adjoint representation} or the regular representation of $ (P, ~ \!  \cdot ~ \! , \{ ~, ~ \} )$.

For any Poisson algebra $ (P, ~ \!  \cdot ~ \! , \{ ~, ~ \} )$, the triple $(P^*, \mu_\mathrm{coad}, \rho_\mathrm{coad})$ is also a representation, where $(\mu_\mathrm{coad})_x (\zeta) (y) = \langle \zeta, x \cdot y \rangle$ and $(\rho_\mathrm{coad})_x (\zeta) (y) = - \langle \zeta, \{x, y \} \rangle$, for all $x, y \in P$ and $\zeta \in P^*$. This is called the {\em coadjoint representation}. In general, given any representation of a Poisson algebra, one can dualize the representation (see for example, \cite{ni-bai}). With this, the coadjoint representation is simply the dual of the adjoint representation.

\medskip

Let $ (P, ~ \!  \cdot ~ \! , \{ ~, ~ \} )$ be a Poisson algebra and $(V, \mu, \rho)$ be a representation of it. A linear map $f \in \mathrm{Hom} (P^{\otimes m} \otimes \wedge^n P, V)$ is called a {\em $(m,n)$-cochain} if
\begin{align*}
    \sum_{\sigma \in \mathrm{Sh}_{(i, m-i)}} (-1)^\sigma ~ f \big(  ( a_{\sigma^{-1} (1)} \otimes \cdots \otimes a_{\sigma^{-1} (i)} \otimes a_{\sigma^{-1} (i+1)} \otimes \cdots \otimes a_{\sigma^{-1} (m)} ) \otimes (x_1 \wedge \cdots \wedge x_n )   \big) = 0,
\end{align*}
for all $0 < i < m$ and $a_1, \ldots, a_m, x_1, \ldots, x_n \in P$. We denote the set of all $(m, n)$-cochains by $C^{m,n} (P, V)$. Then for any $k \geq 0$, the space of $k$-cochains is defined by $ C^k (P, V) := \bigoplus_{\substack{m+n = k \\ m \neq 1}} C^{m, n} (P, V)$
and the coboundary map $\delta_\mathrm{FGV} :  C^k (P, V)  \rightarrow  C^{k+1} (P, V)$ is given by
\begin{align*}
    \delta_\mathrm{FGV} (f) = \delta_\mathrm{H}^{m,n} (f) + (-1)^m ~ \! \delta_\mathrm{CE}^{m,n} (f), 
\end{align*}
for $ f \in C^{m,n}(P, V)$ with $m + n = k ~ (m \neq 1)$. Here the map $\delta_\mathrm{H}^{m,n} : C^{m,n} (P, V) \rightarrow C^{m+1, n} (P, V)$ is the Harrison coboundary operator
\begin{align*}
     ( \delta_\mathrm{H}^{m,n} f ) ( (a_1 \otimes \cdots \otimes a_{m+1}) ~\otimes~& (x_1 \wedge \cdots \wedge x_n) ) 
     = \mu_{a_1} f ( (a_2 \otimes \cdots \otimes a_{m+1})\otimes (x_1 \wedge \cdots \wedge x_n) ) \\
     &+ \sum_{i=1}^m (-1)^i f (  (a_1 \otimes \cdots \otimes a_i \cdot a_{i+1} \otimes \cdots \otimes a_{m+1}) \otimes (x_1 \wedge \cdots \wedge x_n) ) \\
     &+ (-1)^{m+1} ~ \!  \mu_{a_{m+1}}f ( (a_1 \otimes \cdots \otimes a_{m})\otimes (x_1 \wedge \cdots \wedge x_n) )
\end{align*}
and $\delta_\mathrm{H}^{0, n} : C^{0, n} (P, V) \rightarrow C^{2, n-1} (P, V)$ is the composition of the natural inclusion
\begin{align*}
    C^{0, n} (P, V) = \mathrm{Hom} (\wedge^n P, V) \hookrightarrow \mathrm{Hom} (P \otimes \wedge^{n-1} P, V) = C^{1, n-1} (P, V)
\end{align*}
with the Harrison coboundary operator $\delta_\mathrm{H}^{1, n-1} : C^{1, n-1} (P, V) \rightarrow C^{2, n-1} (P, V)$. On the other hand, the map $\delta^{m,n}_\mathrm{CE} : C^{m,n} (P, V) \rightarrow C^{m, n+1} (P, V)$ is the Chevalley-Eilenberg coboundary operator
\begin{align*}
    (\delta^{m,n}_\mathrm{CE} f) & (  (a_1 \otimes \cdots \otimes a_{m})\otimes (x_1 \wedge \cdots \wedge x_{n+1})) \\
    &= \sum_{i=1}^{n+1} (-1)^{i+1} \bigg(  \rho_{x_i} f ( (a_1 \otimes \cdots \otimes a_{m}) \otimes (x_1 \wedge \cdots \wedge \widehat{x_i} \wedge \cdots  \wedge x_{n+1}) )  \\
    & \quad - \sum_{j=1}^m f (  (a_1 \otimes \cdots \otimes \{ x, a_j \} \otimes \cdots \otimes a_{m}) \otimes (x_1 \wedge \cdots \wedge \widehat{x_i} \wedge \cdots  \wedge x_{n+1})    )    \bigg) \\
    & \quad + \sum_{1 \leq i < j \leq n+1} (-1)^{i+j} f (   (a_1 \otimes \cdots \otimes a_{m}) \otimes ( \{ x_i, x_j \} \wedge x_1 \wedge \cdots \wedge \widehat{x_i} \wedge \cdots \wedge \widehat{x_j} \wedge \cdots \wedge x_{n+1})   ),
\end{align*}
for $f \in C^{m,n} (P, V)$ and $a_1, \ldots, a_{m}, x_1, \ldots, x_{n+1} \in P$. Then it turns out that $\{ C^\bullet (P, V), \delta_\mathrm{FGV} \}$ is a cochain complex, called the {\em Flato-Gerstenhaber-Voronov (FGV) cochain complex}. The set of all $n$-cocycles is denoted by $Z^n (P, V)$ and the set of all $n$-coboundaries is denoted by $B^n (P, V)$. The corresponding cohomology groups are called the (FGV) cohomology of the Poisson algebra $(P, ~ \! \cdot ~ \!, \{ ~,~ \})$ with coefficients in the representation $(V, \mu, \rho)$, and they are simply denoted by $H^\bullet_\mathrm{FGV} (P, V)$.

It follows from the above discussions that a Poisson $2$-cocycle can be described by a pair $(h, H)$ consisting of a symmetric map $h : P^{\otimes 2} \rightarrow V$ and a skew-symmetric map $H : \wedge^2 P \rightarrow V$ that satisfy the following conditions:
\begin{align}\label{pois-co1}
 \mu_x h (y, z) - h(x \cdot y, z) + h (x, y \cdot z) - \mu_z h (x, y) = 0,
\end{align}
\begin{align}\label{pois-co2}
    \rho_x H (y, z) + \rho_y H (z, x) + \rho_z H (x, y) + H (x, \{ y, z\}) + H (y, \{ z, x \}) + H (z, \{x, y \}) = 0,
\end{align}
\begin{align}\label{pois-co3}
     \rho_x h (y, z) - h (\{ x, y \}, z) - h (y, \{ x, z \} ) + H (x, y \cdot z ) - \mu_y H (x, z) - \mu_z H (x, y)  = 0,
\end{align}
for all $x, y, z \in P.$ Further, two Poisson $2$-cocycles $(h, H)$ and $(h', H')$ are {\em cohomologous} if there exists a linear map $\varphi : P \rightarrow V$ such that for all $x, y \in P,$
\begin{align}\label{cobound}
     h(x, y) - h' (x, y) = \mu_x \varphi (y)- \varphi (x \cdot y ) + \mu_y \varphi (x), \quad H (x, y) - H' (x, y) = \rho_x \varphi (y) - \varphi (\{x, y \}) - \rho_y \varphi (x).
\end{align}

\section{Inducibility of Poisson automorphisms}\label{sec3}
In this section, we study the inducibility problem for a pair of Poisson algebra automorphisms in a given abelian extension of a Poisson algebra by a representation. To find an answer to this problem, we define the Wells map in the context of Poisson algebras and show that a pair of Poisson algebra automorphisms is inducible if and only if its image under the Wells map (the image lies in the second cohomology group) vanishes identically.

\medskip

Let 
\begin{align}\label{abelian-aut}
\xymatrixrowsep{0.36cm}
\xymatrixcolsep{0.36cm} \xymatrix{
    0 \ar[r] & V \ar[r]^i & E \ar[r]^p & P \ar[r] & 0
    }
    \end{align}
    be an abelian extension of the Poisson algebra $(P, ~ \! \cdot ~ \!, \{ ~, ~ \})$ by a representation $(V, \mu, \rho)$.
 Let $\mathrm{Aut}_{V} (E)$ be the group of all Poisson algebra automorphisms $\gamma \in \mathrm{Aut} (E)$ that satisfies $\gamma ({V}) \subset V$. Hence, for any $\gamma \in \mathrm{Aut}_{V} (E)$ we naturally have $\gamma \big|_{V} \in \mathrm{Aut} (V)$ a Poisson algebra automorphism. Next, for any linear section $s: P \rightarrow E$ of the map $p$ (i.e., $p \circ s = \mathrm{id}_{P}$), we can also define a map $\overline{\gamma} : P \rightarrow P$ by $\overline{\gamma} (x) := p \gamma s (x)$, for $x \in P$. It can be easily checked that the map $\overline{\gamma}$ doesn't depend on the choice of the section $s$. Moreover, $\overline{\gamma}$ is a bijection on the set $P$. Let $( h, H)$ be the Poisson $2$-cocycle corresponding to the given abelian extension and induced by the section $s$ \cite{agore}. Then 
\begin{align*}
    h (x, y) := s(x) \cdot_E s(y) - s (x \cdot y) \quad  \text{ and } \quad H (x, y) = \{ s(x) , s(y) \}_E - s \{ x, y \}, \text{ for } x, y \in P.
\end{align*}
Hence we get 
\begin{align*}
    \overline{\gamma} (x \cdot y) = p \gamma ( s (x \cdot y)) =~& p \gamma \big(  s(x) \cdot_E s (y) - h (x, y) \big) \\
    =~& p \gamma \big(  s(x) \cdot_E s (y)  \big) \quad (\because ~ \gamma (V) \subset V ~~ \text{ and }~~ p \big|_{V} = 0) \\
   =~& p \gamma s (x) \cdot p \gamma s (y) = \overline{\gamma} (x) \cdot \overline{\gamma} (y).
\end{align*}
In the same way, we can show that $\overline{\gamma} ( \{ x, y \}) = \big\{ \overline{\gamma} (x), \overline{\gamma} (y) \}$, for all $x, y \in P$. Thus, $\overline{\gamma} : P \rightarrow P$ is a Poisson algebra automorphism, i.e., $\overline{\gamma} \in \mathrm{Aut} (P)$. Hence, we obtain a group homomorphism 
\begin{align*}
    \tau : \mathrm{Aut}_{V} (E) \rightarrow \mathrm{Aut} (V) \times \mathrm{Aut} (P) ~~ \text{ given by }~~ \tau (\gamma) := (\gamma \big|_{V}, \overline{\gamma}).
\end{align*}
Keeping this in mind, we say that a pair of Poisson algebra automorphisms $(\beta, \alpha) \in \mathrm{Aut} (V) \times \mathrm{Aut} (P)$ is {\bf inducible} if this pair lies in the image of the map $\tau$. Equivalently, $(\beta, \alpha)$ is inducible if there exists a Poisson algebra automorphism $\gamma \in \mathrm{Aut}_{V} (E)$ such that $\gamma \big|_{V} = \beta$ and $\overline{\gamma} = \alpha$.

Our aim in this section is to find a necessary and sufficient condition (independent of the choice of any section) for a pair of Poisson algebra automorphisms in $\mathrm{Aut} (V) \times \mathrm{Aut} (P)$ to be inducible. We will start with the following result.

\begin{proposition}\label{prop-ind-first}
Let (\ref{abelian-aut}) be an abelian extension of the Poisson algebra $(P, ~ \! \cdot ~ \!, \{ ~, ~ \})$ by a representation $(V, \mu, \rho)$. If the pair $(\beta, \alpha) \in \mathrm{Aut} (V) \times \mathrm{Aut} (P)$ of Poisson algebra automorphisms is inducible then
\begin{align}\label{c-iden}
    \beta (\mu_x v ) = \mu_{\alpha (x)} \beta (v) ~~~~~ \text{ and } ~~~~~ \beta (\rho_x v) = \rho_{\alpha (x)} \beta (v), \text{ for any } x \in P, ~ \! \! v \in V.
\end{align}
\end{proposition}

\begin{proof}
    Since $(\beta, \alpha)$ is inducible, there exists a Poisson algebra automorphism $\gamma \in \mathrm{Aut}_V (E)$ such that $\gamma |_V = \beta$ and $\overline{\gamma} = \alpha$. Hence for any $x \in P$ and $v \in V$,
    \begin{align*}
         &\beta (\mu_x v) = \gamma (s(x) \cdot_E v) = \gamma s (x) \cdot_E \gamma (v)         = s \alpha (x) \cdot_E \beta (v)  = \mu_{\alpha (x)} \beta (v), \\
        & \beta (\rho_x v) = \gamma \{ s(x) , v \}_E = \{ \gamma s (x) , \gamma (v) \}_E  = \{s \alpha (x), \beta (v) \}_E  = \rho_{\alpha (x)} \beta(v).
    \end{align*}
    In both the above calculations, we have used that $(\gamma s - s \alpha) (x) \in \mathrm{ker} (p) = \mathrm{im } (i) \cong V$ and $V$ has trivial Poisson structure.
\end{proof}

Given an abelian extension (\ref{abelian-aut}), we set
\begin{align*}
    \mathcal{C}_{\mu, \rho} = \{ (\beta, \alpha) \in \mathrm{Aut} (V) \times \mathrm{Aut} (P) ~ \! | ~ \! \text{the identities in } (\ref{c-iden}) \text{ hold} \}.
\end{align*}
Then it follows that $\mathcal{C}_{\mu, \rho}$ is a subgroup of $ \mathrm{Aut} (V) \times \mathrm{Aut} (P)$. Take an arbitrary element $(\beta, \alpha) \in \mathcal{C}_{\mu, \rho}$. For any linear maps $h : P^{\otimes 2} \rightarrow V$ and $H : \wedge^2 P \rightarrow V$, we define maps $h_{(\beta, \alpha)} : P^{\otimes 2} \rightarrow V$ and $H_{(\beta, \alpha)}  : \wedge^2 P \rightarrow V$ respectively by
\begin{align*}
    h_{(\beta, \alpha)} (x, y) := \beta \big( h ( \alpha^{-1} (x), \alpha^{-1} (y)) \big), \quad H_{(\beta, \alpha)} (x, y) := \beta \big( H ( \alpha^{-1} (x), \alpha^{-1} (y)) \big), 
\end{align*}
for $x, y \in P.$ Then we have the following.

\begin{proposition}\label{prop-2co}
    \begin{itemize}
   \item[(i)] If $(h, H)$ is a Poisson $2$-cocycle of the Poisson algebra $(P, ~ \! \cdot ~ \!, \{ ~, ~ \})$ with coefficients in the representation $(V, \mu, \rho)$ then the pair $(h_{(\beta, \alpha)}, H_{(\beta, \alpha)})$ is so.

   \item[(ii)] If $(h, H)$ are $(h', H')$ are cohomologous Poisson $2$-cocycles then $(h_{(\beta, \alpha)}, H_{(\beta, \alpha)})$ and $(h'_{(\beta, \alpha)}, H'_{(\beta, \alpha)})$ are also cohomologous.
    \end{itemize}
\end{proposition}

\begin{proof}
    (i) Note that $( h, H)$ is a  Poisson $2$-cocycle implies that $h$ is symmetric, $H$ is skew-symmetric and the identities (\ref{pois-co1})-(\ref{pois-co3}) are hold. It is easy to see that the map $h_{(\beta, \alpha)} $ is symmetric and $H_{(\beta, \alpha)} $ is skew-symmetric. Moreover, in all the identities (\ref{pois-co1})-(\ref{pois-co3}), if we replace the elements $x, y, z$ by the elements $\alpha^{-1} (x), \alpha^{-1} (y), \alpha^{-1} (z)$ respectively, we obtain the corresponding identities for the pair $( h_{(\beta, \alpha)}, H_{(\beta, \alpha)} )$. Hence $( h_{(\beta, \alpha)}, H_{(\beta, \alpha)} )$ is a Poisson $2$-cocycle.

    \medskip

    (ii) Since $(h, H)$ and $(h', H')$ are cohomologous, there exists a linear map $\varphi : P \rightarrow V$ such that the identities in (\ref{cobound}) hold. Hence we have
    \begin{align*}
        h_{(\beta, \alpha)} (x, y) - h'_{(\beta, \alpha)} (x, y) =~& \beta \big( \mu_{\alpha^{-1} (x)} \varphi \alpha^{-1} (y) - \varphi (\alpha^{-1} (x) \cdot \alpha^{-1} (y)) +  \mu_{\alpha^{-1} (y)} \varphi \alpha^{-1} (x)    \big) \\
        =~&  \mu_x \beta \varphi \alpha^{-1} (y) - \beta \varphi \alpha^{-1} (x \cdot y) + \mu_y \beta\varphi \alpha^{-1} (x) \quad (\text{by } (\ref{c-iden}))
    \end{align*}
    and similarly, $H_{(\beta, \alpha)} (x, y) - H'_{(\beta, \alpha)} (x, y) = \rho_x \beta \varphi \alpha^{-1} (y) - \beta \varphi \alpha^{-1} (\{ x, y \} ) - \rho_y \beta \varphi \alpha^{-1} (x)$, for all $x, y \in P$. This shows that $(h_{(\beta, \alpha)}, H_{(\beta, \alpha)})$ and $(h'_{(\beta, \alpha)}, H'_{(\beta, \alpha)})$ are cohomologous Poisson $2$-cocycles.
\end{proof}

It follows from the above proposition that there is a well-defined map $ \Psi: \mathcal{C}_{\mu, \rho} \rightarrow \mathrm{End} (H^2_\mathrm{FGV} (P, V))$ given by
\begin{align*}
   \Psi (\beta, \alpha) [(h, H)] := [ (h_{(\beta, \alpha)}, H_{(\beta, \alpha)}) - (h, H)], \text{ for } (\beta, \alpha) \in \mathcal{C}_{\mu, \rho} \text{ and } [(h, H)] \in H^2_\mathrm{FGV} (P, V).
\end{align*}

Given any section $s$, let the abelian extension (\ref{abelian-aut}) produce the Poisson $2$-cocycle $(h, H)$. Then we define a map $\mathcal{W} : \mathcal{C}_{\mu, \rho} \rightarrow H^2_\mathrm{FGV} (P, V)$, called the {\bf Wells map}, given by
\begin{align}\label{wells-aut-map}
    \mathcal{W} (\beta, \alpha) := \Psi (\beta, \alpha) [(h, H)] = [ (h_{(\beta, \alpha)}, H_{(\beta, \alpha)}) - (h, H)], \text{ for } (\beta, \alpha) \in \mathcal{C}_{\mu, \rho}.
\end{align}
It follows from Proposition \ref{prop-2co} (ii) that the Wells map doesn't depend on the choice of the section $s$. We are now ready to prove the main result of this section.

\begin{theorem}
    Let (\ref{abelian-aut}) be an abelian extension of the Poisson algebra $(P, ~ \! \cdot ~ \!, \{ ~, ~ \})$ by a representation $(V, \mu, \rho)$. Then a pair $(\beta, \alpha ) \in \mathrm{Aut} (V) \times \mathrm{Aut}(P)$ of Poisson algebra automorphisms is inducible if and only if 
    \begin{align*}
        (\beta, \alpha ) \in \mathcal{C}_{\mu, \rho} \quad \text{ and } \quad \mathcal{W} (\beta, \alpha) = 0.
    \end{align*}
\end{theorem}

\begin{proof}
    Let $(\beta, \alpha) \in \mathrm{Aut} (V) \times \mathrm{Aut}(P)$ be an inducible pair of Poisson algebra automorphisms. Then we have seen in Proposition \ref{prop-ind-first} that $(\beta, \alpha) \in \mathcal{C}_{\mu, \rho}$. Next, consider a Poisson automorphism $\gamma \in \mathrm{Aut}_V (E)$ such that $\gamma |_V = \beta$ and $p \gamma s = \alpha$, where $s$ is any section. Depending on a fixed section $s$, let the given abelian extension produce the Poisson $2$-cocycle $(h, H)$. Then for any $x, y \in P$, we have
    \begin{align*}
       & h_{(\beta, \alpha)} (x, y) - h (x, y) \\
       &= \beta \big(  h (\alpha^{-1} (x), \alpha^{-1} (y)) \big) - h (x, y) \\
       &= \beta \big( s \alpha^{-1} (x) \cdot_E s \alpha^{-1} (y) - s (\alpha^{-1} (x) \cdot \alpha^{-1} (y) )     \big) - \big( s(x) \cdot_E s (y) - s (x \cdot y)   \big)\\
        &= \gamma s \alpha^{-1} (x) \cdot_E \gamma s \alpha^{-1} (y) - \gamma s (\alpha^{-1} (x) \cdot \alpha^{-1} (y)) - s(x) \cdot_E s(y) + s (x \cdot y)\\
        &=  \gamma s \alpha^{-1} (x) \cdot_E \gamma s \alpha^{-1} (y) - \gamma s (\alpha^{-1} (x) \cdot \alpha^{-1} (y)) - s(x) \cdot_E s(y) + s (x \cdot y) \\
        & \quad - (\gamma s \alpha^{-1} - s) (x) \cdot_E (\gamma s \alpha^{-1} - s) (y) \quad (\because ~\!  \mathrm{Im} (\gamma s \alpha^{-1} - s) \subset V \text{ and } V \text{ has trivial Poisson structure}) \\
        &= s (x) \cdot_E \gamma s \alpha^{-1} (y) - s(x) \cdot_E s (y) - \gamma s \alpha^{-1} (x \cdot y) + s (x \cdot y) + \gamma s \alpha^{-1} (x) \cdot_E s(y) - s (x) \cdot_E s (y) \\
        & \qquad \qquad \qquad \qquad \qquad \qquad  (\text{after rearranging}) \\
        &= \mu_x (\gamma s \alpha^{-1} - s) (y) - (\gamma s \alpha^{-1} - s) (x \cdot y) + \mu_y (\gamma s \alpha^{-1} - s) (x),
    \end{align*}
    and similarly, $H_{(\beta, \alpha)} (x, y) - H (x, y) = \rho_x (\gamma s \alpha^{-1} - s) (y) - (\gamma s \alpha^{-1} - s) (\{ x, y \}) -  \rho_y (\gamma s \alpha^{-1} - s) (x) $. This shows that the Poisson $2$-cocycles $(h_{(\beta, \alpha)} , H_{(\beta, \alpha)})$ and $(h, H)$ are cohomologous. Hence we have $\mathcal{W} (\beta, \alpha) = [ (h_{(\beta, \alpha)} , H_{(\beta, \alpha)}) - (h,H)] = 0.$

    \medskip

    Conversely, suppose that $(\beta, \alpha) \in \mathcal{C}_{\mu, \rho}$ and $\mathcal{W} (\beta, \alpha ) = 0$. Let $s$ be any fixed section. Depending on $s$, let the given abelian extension produce the Poisson $2$-cocycle $(h, H)$. Since $\mathcal{W} (\beta, \alpha) = 0$, it follows that the Poisson $2$-cocycles $(h_{(\beta, \alpha)}, H_{(\beta, \alpha)})$ and $(h, H)$ are cohomologous. Thus, there exists a linear map $\varphi : P \rightarrow V$ such that 
    \begin{align}
        \beta \big( h (\alpha^{-1} (x), \alpha^{-1} (y)) \big) - h (x, y) =~& \mu_x \varphi (y) - \varphi (x \cdot y) + \mu_y \varphi (x), \label{cohomo-1}\\
        \beta \big(  H (\alpha^{-1} (x), \alpha^{-1} (y))  \big)- H (x, y) =~& \rho_x \varphi (y) - \varphi (\{ x, y \}) - \rho_y \varphi (x),
    \end{align}
    for all $x, y \in P$. Since $s$ is a section, any element $e \in E$ can be uniquely written as $e = s(x) + u $, for some $x \in P$ and $u \in V$. We now define a map $\gamma : E \rightarrow E$ by
    \begin{align*}
        \gamma (e) = \gamma (s (x) + u) = s (\alpha (x)) + \beta (u) + \varphi (x), \text{ for } e = s(x) + u \in E.
    \end{align*}
    The map is injective as $\gamma (e) = 0$ (for $e = s(x) + u$) implies that $s (\alpha (x)) = 0$ and $\beta (u) + \varphi (x) = 0$. Since $s, \alpha, \beta$ are all injective maps, we get that $x= 0$ and $u = 0$ which in turn implies that $e = 0$. The map $\gamma$ is also surjective as $e = s(x) + u \in E$ has unique preimage $e' = s (\alpha^{-1} (x)) + ( \beta^{-1} (u) - \beta^{-1} \varphi \alpha^{-1} (x)) \in E$. This shows that $\gamma$ is indeed a bijective map. Next, we claim that $\gamma: E \rightarrow E$ is a Poisson algebra automorphism. For any two elements $e = s(x) + u$ and $e' = s(y) + v$ from the space $E$, we observe that 
    \begin{align*}
        &\gamma (e) \cdot_E \gamma (e') \\
        &= \big(  s (\alpha (x)) + \beta (u) + \varphi (x) \big) \cdot_E \big(  s (\alpha (y)) + \beta (v) + \varphi (y)   \big) \\
       & = s (\alpha (x)) \cdot_E s (\alpha (y)) + s (\alpha (x)) \cdot_E \beta (v) + s (\alpha (x)) \cdot_E \varphi(y) + s (\alpha (y)) \cdot_E \beta (u) + s (\alpha (y)) \cdot_E \varphi (x) \\
       & = s (\alpha (x) \cdot \alpha (y)) + h (\alpha (x), \alpha (y)) + \mu_{\alpha (x)} \beta (v) + \mu_{\alpha (x)} \varphi (y) + \mu_{\alpha (y)} \beta (u) + \mu_{\alpha (y)} \varphi (x) \\
        &= s (\alpha (x) \cdot \alpha (y)) + \beta \big(   h (x, y) + \mu_x v + \mu_y u \big) + \varphi (x \cdot y) \qquad (\text{by } (\ref{cohomo-1}) \text{ and } (\ref{c-iden})) \\
       & =\gamma \big(  s (x \cdot y) + h (x, y) + s(x) \cdot_E v + s(y) \cdot_E u  \big) \qquad (\text{from the definition of } \gamma)\\
       & = \gamma \big(  s(x) \cdot_E s (y) + s (x) \cdot_E v + s(y) \cdot_E u  \big) \\
       & = \gamma (( s(x) + u ) \cdot_E (s(y) + v) ) = \gamma (e \cdot_E e').
    \end{align*}
    In the same way, one can show that $\{ \gamma (e) , \gamma (e') \}_E = \gamma \{ e, e' \}_E.$ Hence the claim follows. Finally, it follows from the definition of $\gamma$ that $\gamma (u) = \beta (u)$ and $\overline{\gamma} (x) = p \gamma s (x) = p ( s (\alpha (x)) + \varphi (x) )= \alpha (x)$, for all $x \in P$ and $u \in V$. Hence $\gamma \in \mathrm{Aut}_V (P)$ and $\tau (\gamma) = (\gamma|_V, \overline{\gamma}) = (\beta, \alpha)$ which shows that the pair $(\beta, \alpha)$ is inducible.
\end{proof}




The Wells map given in (\ref{wells-aut-map}) generalizes the classical Wells map defined in the context of abstract groups \cite{wells} (see also \cite{jin-liu,passi}) and Lie algebras \cite{bar-singh}. In those classical cases, the Wells map fits into a short exact sequence. In the context of Poisson algebras, we have a similar generalization stated below. The proof is also similar to the classical cases.

\begin{theorem}
    Let $0 \rightarrow V \xrightarrow{i} E \xrightarrow{p} P \rightarrow 0$ be an abelian extension of the Poisson algebra $(P, ~ \! \cdot ~ \! , \{ ~, ~ \})$ by a given representation $(V, \mu, \rho)$. Then there is an exact sequence
    \begin{align}\label{wes-aut}
        1  \rightarrow \mathrm{Aut}_{V, P} (E) \xrightarrow{ \iota} \mathrm{Aut}_{V} (E) \xrightarrow{\tau} \mathcal{C}_{\mu, \rho} \xrightarrow{\mathcal{W}} H^2_\mathrm{FGV} (P, V),
    \end{align}
    where $\mathrm{Aut}_{V,P} (E) = \{ \gamma \in \mathrm{Aut}_{V} (E) ~ \! | ~ \!  \tau (\gamma) = (\mathrm{id}_{V}, \mathrm{id}_{P}) \}$.
\end{theorem}

\medskip

An abelian extension (\ref{abelian-aut}) is said to be {\bf split} if there exists a section $s : P \rightarrow E$ which is also a homomorphism of Poisson algebras. In this case, we can identify $E \cong P \oplus V$ as Poisson algebras, where the Poisson algebra structure on $P \oplus V$ is given by
\begin{align*}
    (x, u) \cdot_\ltimes (y, v) = (x \cdot y ~ \! , ~ \! \mu_x v + \mu_y u) ~~~~ \text{ and } ~~~~ \{ (x, u) , (y, v) \}_\ltimes = (\{ x, y \} ~ \!, ~ \! \rho_x v - \rho_y u),
\end{align*}
 for $(x, u), (y, v) \in P \oplus V.$ With this identification, the maps $i, p$ and $s$ are the obvious ones. For the above section $s$, if the abelian extension produce the Poisson $2$-cocycle $(h, H)$, then
\begin{align*}
    h (x, y) = s(x) \cdot_E s(y) - s (x \cdot y) = 0 \quad \text{ and } \quad H (x, y) = \{ s(x), s(y) \}_E - s \{ x, y \} = 0.
\end{align*}
Hence $(h, H) = 0$ which in turn implies that the Wells map $\mathcal{W}$ vanishes identically. Thus the Wells exact sequence (\ref{wes-aut}) takes the form
\begin{align}\label{wes-new-aut}
    1 \rightarrow \mathrm{Aut}_{V, P} (E) \rightarrow \mathrm{Aut}_V (E) \xrightarrow{\tau} \mathcal{C}_{\mu, \rho} \rightarrow 1.
\end{align}
Let $(\beta, \alpha) \in \mathcal{C}_{\mu, \rho}$. Then we define a map $\gamma_{(\beta, \alpha)} : E \rightarrow E$ by $\gamma_{(\beta, \alpha)} (x, u) = (\alpha (x), \beta (u))$, for $(x, u) \in E$. It is easy to see that $\gamma_{(\beta, \alpha)} \in \mathrm{Aut}_V (E)$. Hence there is a map $t : \mathcal{C}_{\mu, \rho} \rightarrow \mathrm{Aut}_V (E)$ given by $t (\beta, \alpha) := \gamma_{(\beta, \alpha)}$, for $(\beta, \alpha) \in \mathcal{C}_{\mu, \rho}$. The map $t$ is a group homomorphism satisfying additionally $\tau t = \mathrm{Id}_{\mathcal{C}_{\mu, \rho}}$. Thus, (\ref{wes-new-aut}) is a split short exact sequence of groups that yields the following isomorphism of groups
\begin{align*}
    \mathrm{Aut}_V (E) \cong \mathcal{C}_{\mu, \rho} \ltimes \mathrm{Aut}_{V, P} (E).
\end{align*}

\medskip

\section{Inducibility of Poisson derivations}\label{sec4}

In this section, we study the inducibility of a pair of Poisson derivations in a given abelian extension of a Poisson algebra by a representation. We observe that the image of a suitable Wells-type map can describe the corresponding obstruction for this inducibility problem. In the end, we also obtain a short exact sequence connecting some derivation spaces and the second Poisson cohomology group.

 As before, let (\ref{abelian-aut}) be a fixed abelian extension of a Poisson algebra $(P, ~ \! \cdot ~ \!, \{ ~, ~ \})$ by a representation $(V, \mu, \rho)$. We also let
\begin{align*}
    \mathrm{Der}_{V} (E) = \{ d : E \rightarrow E \text{ is a Poisson derivation } | ~ \! d (V) \subset V \} 
\end{align*}
be the set of all Poisson derivations on $E$ that make the subspace $V$ invariant. For any $d \in \mathrm{Der}_V (E)$, it follows that $d |_V \in \mathrm{Der} (V)$. Moreover, for any linear section $s : P \rightarrow E$, we define a map $\overline{d} : P \rightarrow P$ by $\overline{d} (x) := pd s (x)$, for $x \in P$. The map $\overline{d}$ doesn't depend on the choice of $s$. Further, for any $x, y \in P$, 
\begin{align*}
    \overline{d} (x \cdot y) =~& pd \big(  s(x) \cdot_E s(y) - h (x, y)  \big) \\
    =~& pd ( s(x) \cdot_E s (y)) \quad (\because ~ d (V) \subset V \text{ and } p |_{V} = 0) \\
    =~& p \big( ds (x) \cdot_E s(y) + s(x) \cdot_E ds (y)    \big) = \overline{d} (x) \cdot y + x \cdot \overline{d} (y) \quad (\because ~ ps = \mathrm{id}_P).
\end{align*}
Similarly, one can show that $\overline{d} (\{ x, y \}) = \{ \overline{d} (x), y \} + \{ x, \overline{d} (y) \}$, for all $x, y \in P$. This shows that the map $\overline{d}$ is a Poisson derivation on the Poisson algebra $P$. Hence, there is a well-defined map 
\begin{align*}
\eta : \mathrm{Der}_{V} (E) \rightarrow \mathrm{Der} (V) \times \mathrm{Der} (P) ~~ \text{ given by } ~~ \eta (d) = (d |_{V}, \overline{d}).
\end{align*}
A pair of Poisson derivations $(d_V, d_P) \in \mathrm{Der} (V) \times \mathrm{Der} (P)$ is said to be {\bf inducible} if there exists a Poisson derivation $d \in \mathrm{Der}_V (E)$ such that $\eta (d) = (d_V, d_P)$. Our aim in this section is to find a necessary and sufficient condition under which a pair of Poisson derivations is inducible. For this, we will present a Wells-type map in the present context and show that a pair of Poisson derivations is inducible if and only if its image under this new Wells-type map vanishes identically. We start with the following result.

\begin{proposition}\label{some-prop}
Let (\ref{abelian-aut}) be an abelian extension of the Poisson algebra $(P, ~ \! \cdot ~ \!, \{ ~, ~ \})$ by a representation $(V, \mu, \rho)$. If the pair of Poisson derivations $(d_V, d_P) \in \mathrm{Der} (V) \times \mathrm{Der} (P)$ is inducible then 
    \begin{align}\label{g1}
        d_V (\mu_x v) = \mu_{d_P (x)} v + \mu_x d_V (v)  \quad \text{ and } \quad d_V (\rho_x v) = \rho_{d_P (x)} v + \rho_x d_V (v), \text{ for any } x \in P, ~ \! v \in V.
    \end{align}
\end{proposition}

\begin{proof}
    Since $(d_V, d_P)$ is an inducible pair, there exists a Poisson derivation $d \in \mathrm{Der}_V (E)$ such that $d \big|_V = d_V$ and $pds = d_P$ (here $s$ is any arbitrary section). Hence for any $x \in P$ and $v \in V$,
    \begin{align*}
        &d_V (\mu_x v)    && d_V (\rho_x v ) \\
        &= d (s(x) \cdot_E v)   &&= d \{ s(x) , v \}_E  \\
        &= ds (x) \cdot_E v + s (x) \cdot_E d(v) && = \{ ds(x), v \}_E + \{ s(x), d(v) \}_E \\
        &=sd_P (x) \cdot_E v + s(x) \cdot_E d_V (v) && = \{ s d_P (x), v \}_E + \{ s(x), d_V (v) \}_E \\
        &= \mu_{d_P (x)} v + \mu_x d_V (v), && = \rho_{d_P (x)} v + \rho_x d_V (v).
    \end{align*}
In the third equality of both the above calculations, we have used that $(ds - sd_P) (x) \in \mathrm{ker} (p) = \mathrm{im} (i) \cong V.$ This proves the result.
\end{proof}

Given an abelian extension as above, we define
\begin{align*}
    \mathcal{D}_{\mu, \rho} := \big\{  (d_V, d_P ) \in  \mathrm{Der} (V) \times \mathrm{Der} (P) ~ \! | ~ \! \text{the identities in } (\ref{g1}) \text{ hold} \big\}.
\end{align*}
It is easy to verify that $\mathcal{D}_{\mu, \rho}$ has a Lie algebra structure with the componentwise commutator bracket. Let $(d_V, d_P) \in \mathcal{D}_{\mu, \rho}$ be any arbitrary element. For any linear maps $h : P^{\otimes 2} \rightarrow V$ and $H : \wedge^2 P \rightarrow V$, we define new maps $h_{(d_V, d_P ) } : P^{\otimes 2} \rightarrow V$ and $H_{ (d_V, d_P )} : \wedge^2 P \rightarrow V$ respectively by
\begin{align*}
    h_{(d_V, d_P )} (x, y) :=~& d_V ( h (x, y)) - h (d_P (x), y) - h (x, d_P (y)),\\
    H_{(d_V, d_P )} (x, y) :=~& d_V ( H (x, y)) - H (d_P (x), y) - H (x, d_P (y)),
\end{align*}
for $x, y \in P$. With these notations, we have the following.

\begin{proposition}\label{prop-two}
\begin{itemize}
    \item[(i)]  If $(h, H)$ is a Poisson $2$-cocycle of the Poisson algebra $(P, ~ \! \cdot ~ \!, \{ ~ , ~ \})$ with coefficients in the representation $(V, \mu, \rho)$ then $ ( h_{(d_V, d_P )}, H_{(d_V, d_P )} )$ is so.
    \item[(ii)] If $(h, H)$ and $(h', H')$ are two cohomologous Poisson $2$-cocycles then $ ( h_{(d_V, d_P )}, H_{(d_V, d_P )} )$ and $ ( h'_{(d_V, d_P )}, H'_{(d_V, d_P )} )$ are also cohomologous.
\end{itemize}
\end{proposition}

\begin{proof}
  (i) Note that $(h, H)$ is a Poisson $2$-cocycle implies that the identities in (\ref{pois-co1}), (\ref{pois-co2}) and (\ref{pois-co3}) hold. First, we observe that the map $h_{(d_V, d_P )}$ is commutative and for any $x, y, z \in P$,
  \begin{align*}
      &\mu_x h_{(d_V, d_P )} (y, z) - h_{(d_V, d_P )} (x \cdot y, z) + h_{(d_V, d_P )} (x, y \cdot z) - \mu_z h_{(d_V, d_P )} (x, y) \\
      &= \mu_x \big(   d_V (h (y, z)) - h  (d_P (y), z) - h (y, d_P (z))   \big) - d_V (h (x \cdot y, z)) + h (d_P (x \cdot y), z ) + h (x \cdot y, d_P (z)) \\
     & ~~~ + d_V (h (x , y \cdot z)) - h (d_P (x), y \cdot z) - h (x, d_P (y \cdot z)) - \mu_z \big(  d_V (h (x, y)) - h (d_P (x), y) - h (x, d_P (y)) \big) \\
     & = d_V \mu_x h (y, z) - \mu_{d_P (x)} h (y, z) - \mu_x h  (d_P (y), z) - \mu_x h (y, d_P (z)) - d_V h (x \cdot y, z) \\
     & ~~~ + h (d_P (x) \cdot y , z) + h (x \cdot d_P (y), z) + h (x \cdot y, d_P (z)) + d_V h (x, y \cdot z) - h (d_P (x) , y \cdot z) \\
      & ~~~ - h (x, d_P (y) \cdot z) - h (x, y \cdot d_P (z)) - d_V \mu_z h (x, y) + \mu_{d_P (z)} h (x, y) + \mu_z h (d_P (x), y) + \mu_z h (x, d_P (y)) \\
      &= 0 \qquad (\text{as } h \text{ satisfies } (\ref{pois-co1}))
  \end{align*}
  On the other hand, the map $H_{(d_V, d_P)}$ is skew-symmetric and also we have
  \begin{align*}
      &\rho_x H_{(d_V, d_P)} (y, z) + c.p. + H_{(d_V, d_P)} (x, \{ y, z \} ) + c.p. \\
     &= \rho_x \big\{ d_V (H (y, z)) - H (d_P(y), z) - H (y, d_P (z)) \big\} + c. p. \\
& \quad \qquad + \big\{  d_V (H (x, \{ y, z \} )) - H (d_P  (x), \{ y, z \} ) - H (x, d_P \{ y, z \})  \big\} + c.p. \\
&= \big\{  d_V \rho_x H (y, z) - \rho_{d_P  (x)} H (y, z) - \rho_x H (d_P (y), z) - \rho_x H (y, d_P (z))  \big\} + c.p. \\
& \quad \qquad + \big\{ d_V (H (x, \{ y, z \} )) - H (d_P (x), \{ y, z \} ) - H (x, \{ d_P (y), z\} ) - H (x, \{ y, d_P (z)\} )   \big\} + c.p. \\
&= 0 \qquad (\text{as } H \text{ satisfies } (\ref{pois-co2})).
\end{align*}
Here $c.p.$ stands for the cyclic permutations of the inputs $x, y, z$. Finally, we also have
\begin{align*}
   &\rho_x h_{(d_V, d_P)} (y, z) - h_{(d_V, d_P)} (\{ x, y\}, z) - h_{(d_V, d_P)} (y, \{ x, z \}) + H_{(d_V, d_P)} (x, y \cdot z) \\
  & \qquad \qquad  - \mu_y H_{(d_V, d_P)} (x, z)  - \mu_z H_{(d_V, d_P)} (x, y) \\
   &= \rho_x \big(  d_V h (y, z) - h (d_P (y), z) - h (y, d_P (z)) \big) - d_V h (\{x, y \}, z) + h (d_P \{ x, y \}, z) + h (\{x, y \}, d_P (z)) \\
   & ~~~ - d_V h (y, \{x, z\}) + h (d_P (y), \{ x, z \}) + h (y, d_P \{ x, z\}) + d_V H (x, y \cdot z) - H (d_P (x), y \cdot z) - H (x, d_P (y \cdot z)) \\
   & ~~~ - \mu_y \big( d_V H (x, z) - H (d_P (x), z) - H (x, d_P (z))   \big) - \mu_z \big( d_V H (x, y) - H (d_P (x), y) - H (x, d_P (y))   \big) \\
  & = 0 \qquad (\text{by using } (\ref{g1}), \text{ the derivation property of } d_P \text{ and the identity } (\ref{pois-co3})).
\end{align*}
This shows that the pair $(h_{(d_V, d_P)}, H_{(d_V, d_P)})$ satisfy all the identities (\ref{pois-co1}), (\ref{pois-co2}) and (\ref{pois-co3}). Hence it is a Poisson $2$-cocycle.

\medskip

  (ii) Since $(h, H)$ and $(h', H')$ are cohomologous Poisson $2$-cocycles, there exists a linear map $\varphi : P \rightarrow V$ such that the identities in (\ref{cobound}) hold. Hence we have
  \begin{align*}
      &h_{(d_V, d_P)} (x, y) - h'_{(d_V, d_P)} (x, y) \\
      &= d_V (h (x, y)) - h (d_P (x), y) - h (x, d_P(y)) - d_V (h' (x, y)) + h' (d_P (x), y) + h' (x, d_P (y)) \\
      &= d_V \big(  \mu_x \varphi (y) - \varphi (x \cdot y) + \mu_y \varphi (x)  \big) - \big(   \mu_{d_P (x)} \varphi (y) - \varphi ( d_P (x) \cdot y) + \mu_y \varphi (d_P (x))  \big) \\
      & \qquad \qquad - \big( \mu_x \varphi (d_P(y)) - \varphi (x \cdot d_P (y) ) + \mu_{d_P (y)} \varphi (x) \big) \qquad (\text{by } (\ref{cobound}))\\
      &= \mu_x (d_V \circ \varphi - \varphi \circ d_P ) (y) -  (d_V \circ \varphi - \varphi \circ d_P ) (x \cdot y) + \mu_y  (d_V \circ \varphi - \varphi \circ d_P ) (x) \qquad  (\text{by } (\ref{g1}))
  \end{align*}
  and similarly, 
  \begin{align*}
       &H_{(d_V, d_P)} (x, y) - H'_{(d_V, d_P)} (x, y) \\
       & \qquad \qquad = \rho_x (d_V \circ \varphi - \varphi \circ d_P ) (y) - (d_V \circ \varphi - \varphi \circ d_P ) (\{ x, y \}) - \rho_y (d_V \circ \varphi - \varphi \circ d_P ) (x).
  \end{align*}
  This shows that the Poisson $2$-cocycles $ ( h_{(d_V, d_P )}, H_{(d_V, d_P )} )$ and $ ( h'_{(d_V, d_P )}, H'_{(d_V, d_P )} )$ are cohomologous by the map $d_V \circ \varphi - \varphi \circ d_P$. This completes the proof.
\end{proof}

The above proposition shows that there is a well-defined map
\begin{align*}
    \Theta : \mathcal{D}_{\mu, \rho} \rightarrow \mathrm{End} (H^2_\mathrm{FGV} (P, V)) ~~ \text{ given by  } ~~ \Theta (d_V, d_P) [(h, H)] := [ ( h_{(d_V, d_P )}, H_{(d_V, d_P )} ) ],
\end{align*}
for any arbitrary $(d_V, d_P) \in \mathcal{D}_{\mu, \rho}$ and $ [(h, H)] \in H^2_\mathrm{FGV} (P, V)$. For a fixed section $s$, let the given abelian extension (\ref{abelian-aut}) produce the Poisson $2$-cocycle $(h, H)$. Then we define a map (also denoted by the same notation as before) $ \mathcal{W} : \mathcal{D}_{\mu, \rho} \rightarrow H^2_\mathrm{FGV} (P, V)$ by
\begin{align}
\mathcal{W} (d_V, d_P):= \Theta (d_V, d_P) [(h, H)] = [ ( h_{(d_V, d_P )}, H_{(d_V, d_P )} ) ], \text{ for } (d_V, d_P) \in \mathcal{D}_{\mu, \rho}.
\end{align}
The map $\mathcal{W}$ defined above is called the {\bf Wells map} in the context of Poisson derivations. It follows from Proposition \ref{prop-two} (ii) that the above Wells map is independent of the choice of the section $s$. We are now ready to prove the main result of this section.

\begin{theorem}\label{thm-ind-der}
    Let (\ref{abelian-aut}) be an abelian extension of the Poisson algebra $(P, ~ \! \cdot ~ \! , \{ ~, ~ \})$ by a representation $(V, \mu, \rho)$. Then a pair $(d_V, d_P) \in \mathrm{Der} (V) \times \mathrm{Der} (P)$ of Poisson derivations is inducible if and only if 
    \begin{align*}
        (d_V, d_P) \in \mathcal{D}_{\mu, \rho} \quad \text{ and } \quad \mathcal{W} (d_V, d_P) = 0.
    \end{align*}
\end{theorem}

\begin{proof}
    Let $(d_V, d_P) \in \mathrm{Der} (V) \times \mathrm{Der} (P)$ be an inducible pair of Poisson derivations. Then there exists a Poisson derivation $d \in \mathrm{Der}_V (E)$ such that $d |_V = d_V$ and $pds = d_P$, where $s$ is any section. Further, we have seen in Proposition \ref{some-prop} that $(d_V, d_P) \in \mathcal{D}_{\mu, \rho}$. Depending on a fixed section $s$, let the given abelian extension produce the Poisson $2$-cocycle $(h, H)$. Then we have
    \begin{align*}
        &h_{(d_V, d_P)} (x, y) \\
        &= d_V (h (x, y)) - h (d_P (x), y) - h (x, d_P (y)) \\
        &= d_V \big( s(x) \cdot_E s(y) - s (x \cdot y) \big) - sd_P (x) \cdot_E s (y) + s (d_P (x) \cdot y) - s(x) \cdot_E s d_P (y) + s (x \cdot d_P (y)) \\
        &= s (x) \cdot_E ds (y) + ds (x) \cdot_E s (y) - ds (x \cdot y) - s d_P (x) \cdot_E s (y) - s(x) \cdot_E s d_P (y) + sd_P (x \cdot y) \\
        & \qquad \qquad \qquad (\because ~ d_V = d|_V \text{ and } d, d_P \text{ are derivations}) \\
        &= \mu_x (ds - sd_ P)(y) - (ds - sd_P) (x \cdot y) + \mu_y (ds - sd_P) (x)
    \end{align*}
    and similarly, $ H_{(d_V, d_P)} (x, y) = \rho_x (ds - s d_P) (y) - (ds - s d_P) (\{ x, y \} ) - \rho_y (ds - s d_P) (x)$, for all $x, y \in P$. This shows that the Poisson $2$-cocycle $(h_{(d_V, d_P)}, H_{(d_V, d_P)})$ is cohomologous to the null Poisson $2$-cocycle $(0,0)$. Hence $\mathcal{W} (d_V, d_P) = [ (h_{(d_V, d_P)}, H_{(d_V, d_P)})] = 0$.

    \medskip

     Conversely, we assume that $(d_V, d_P) \in \mathcal{D}_{\mu, \rho}$ and $\mathcal{W} (d_V, d_P) = 0$. Depending on a section $s$, let the given abelian extension produce the Poisson $2$-cocycle $(h, H)$. Since $\mathcal{W} (d_V, d_P) = 0$, it follows that the Poisson $2$-cocycle $(h_{(d_V, d_P)} , H_{(d_V, d_P)})$ is cohomologous to the null Poisson $2$-cocycle. That is, there exists a linear map $\varphi : P \rightarrow V$ such that
     \begin{align}
         d_V (h (x, y)) - h (d_P (x), y) - h (x, d_P (y)) =~&  \mu_x \varphi (y) - \varphi (x \cdot y) + \mu_y \varphi (x), \label{h1} \\
         d_V (H (x, y)) - H (d_P (x) , y ) - H (x, d_P (y)) = ~& \rho_x \varphi (y) - \varphi (\{ x, y \}) - \rho_y \varphi (x), \label{h2}
     \end{align}
     for all $x, y \in P$. We now define a linear map $d: E \rightarrow E$ by
     \begin{align}\label{defi-of-d}
         d(e) = d(s(x) + u) = s (d_P (x)) + d_V (u) + \varphi (x), \text{ for } e = s (x) + u \in E.
     \end{align}
     Then for any $e= s(x) + u$ and $e'= s(y) + v$ from the space $E$, we observe that
     \begin{align*}
        & d (e \cdot_E e') = d ((s(x) + u ) \cdot_E (s(y) + v) ) \\
        & =d \big(  s(x) \cdot_E s(y) + s(x) \cdot_E v + s(y) \cdot_E u  \big) \\
        & = d \big(  s (x \cdot y) + h (x, y) + \mu_x v + \mu_y u  \big) \\
         &= s (d_P (x \cdot y)) + d_V h (x, y) + d_V \mu_x v + d_V \mu_y u + \varphi (x \cdot y)  \quad (\text{from the definition of } d)\\
         &= s ( d_P (x) \cdot y + x \cdot d_P (y)) + h (d_P (x), y) + h (x, d_P (y)) + \mu_x \varphi (y) + \mu_y \varphi (x)\\
         & \qquad + \mu_{d_P (x)} v + \mu_x d_V (v) + \mu_{d_P (y)} u + \mu_y d_V (u) \quad  (\text{by } (\ref{h1}) \text{ and } (\ref{g1}))\\
         &= s (d_P (x)) \cdot_E s(y) + s(x) \cdot_E s (d_P (y)) + s(x) \cdot_E \varphi (y) + s(y) \cdot_E \varphi (x) \\
         & \qquad + s (d_P (x)) \cdot_E v + s(x) \cdot_E d_V (v) + s (d_P (y)) \cdot_E u + s(y) \cdot_E d_V (u) \\
         &= \big(  s (d_P (x)) + d_V (u) + \varphi (x)   \big) \cdot_E (s(y) + v) ~ \! + ~ \! (s(x) + u) \cdot_E \big(  s(d_P (y)) + d_V (v) + \varphi (y)   \big)\\
         &= d(e) \cdot_E e' + e \cdot_E d (e')
     \end{align*}
     and along the same way, $d \{ e, e'\}_E = \{ d (e) ,e' \}_E + \{ e, d (e') \}_E$. This shows that $d \in \mathrm{Der} (E)$ is a Poisson derivation on the Poisson algebra $E$. Further, it follows from (\ref{defi-of-d}) that $d (V) \subset V$ (additionally $d |_V = d_V$) and $pds = d_P$. Hence $d \in \mathrm{Der}_V (E)$ and $\eta (d) = (d|_V , pds) = (d_V, d_P)$ which shows that the pair $(d_V, d_P)$ is inducible.
\end{proof}

Next, we show that the Wells map defined above also fits into a short exact sequence. More precisely, we have the following.

\begin{theorem}
   Let (\ref{abelian-aut}) be an abelian extension of the Poisson algebra $(P, ~ \! \cdot ~ \!, \{ ~ , ~ \})$ by a representation $(V, \mu, \rho)$. Then there is a short exact sequence
   \begin{align}\label{ses2}
       0 \rightarrow \mathrm{Der} (P, V) \xrightarrow{\iota} \mathrm{Der}_V (E) \xrightarrow{\eta} \mathcal{D}_{\mu, \rho} \xrightarrow{\mathcal{W}} H^2_\mathrm{FGV} (P, V),
   \end{align}
   where the map $\iota$ is given by $\iota (D) (e) = D( p (e))$, for $D \in \mathrm{Der} (P, V)$ and $e \in E$.
\end{theorem}

\begin{proof}
    The map $\iota$ is injective and hence the sequence is exact in the first place. Next, we take an element $d \in \mathrm{ker} (\eta)$. Thus, $\eta (d) = (d|_V, pds) = (0, 0)$ which means that $d_V = 0$ and $pds = 0$, where $s$ is any arbitrary but fixed section. For any $x \in P$, we observe that $d (s (x)) \in \mathrm{ker} (p) = \mathrm{im } (i) \cong V$. We now define a map $D : P \rightarrow V$ by $D (x) := d (s(x))$, for $x \in P$. Then for any $x, y \in P$,
    \begin{align*}
        D (x \cdot y) = d (s (x \cdot y)) =~& d \big(  s(x) \cdot_E s (y) - h (x, y) \big) \\
        =~& s (x) \cdot_E d s (y) + d (s(x)) \cdot_E s (y) \quad  (\because ~ d \text{ is a Poisson derivation and } d|_V = 0)\\
        =~& \mu_x D (y) + \mu_y D (x)
    \end{align*}
    and in the same way, $D \{ x, y \} = \rho_x D (y) - \rho_y D (x)$. This shows that $D \in \mathrm{Der} (P, V)$ is a Poisson derivation. Finally, for any $e = s(x) + u \in E$, we have
    \begin{align*}
         d(e) = d (s(x) + u) = d(s(x) ) = D (x) = \iota (D) (s(x) + u) = \iota (D) (e)
    \end{align*}
    which in turn implies that $d \in \mathrm{im} (\iota)$. Hence $\mathrm{ker} (\eta) \subset \mathrm{im} (\iota)$. Conversely, if $D \in \mathrm{Der} (P, V)$ then $\eta (\iota (D)) = \big( \iota (D)|_V, p (\iota (D)) s \big) = (0, 0)$ follows from the definition of $\iota 
    (D)$. Thus, we obtain that $\mathrm{im} (\iota) \subset \mathrm{ker} (\eta)$ and hence the sequence (\ref{ses2}) is also exact in the second place. Finally, to show that the sequence is exact in the third place, we begin with an element $(d_V, d_P) \in \mathrm{ker} (\mathcal{W})$. Then it follows from Theorem \ref{thm-ind-der} that the pair $(d_V, d_P)$ is inducible and hence it lies in the image of $\eta$. Conversely, if $(d_V, d_P) \in \mathrm{im} (\eta)$ then it is inducible and thus $\mathcal{W} (d_V, d_P) = 0$. This shows that $\mathrm{ker} (\mathcal{W}) = \mathrm{im} (\eta).$ This completes the proof.
\end{proof}

When (\ref{abelian-aut}) is a split abelian extension of the Poisson algebra $(P, ~\! \cdot ~\!, \{ ~. ~\})$ by the representation $(V, \mu, \rho)$, the Wells map $\mathcal{W} : \mathcal{D}_{\mu, \rho} \rightarrow H^2_\mathrm{FGV} (P, V)$ vanishes identically. Hence, in this case, the short exact sequence (\ref{ses2}) becomes 
\begin{align}\label{ses3}
    0 \rightarrow \mathrm{Der} (P, V) \xrightarrow{ \iota} \mathrm{Der}_V (E) \xrightarrow{ \eta} \mathcal{D}_{\mu, \rho} \rightarrow 0.
\end{align}
For any $(d_V, d_P) \in \mathcal{D}_{\mu, \rho}$, we define a map $d_{(d_V, d_P)} : E \rightarrow E$ by $d_{(d_V, d_P)} (s(x) + u) = s (d_{P} (x)) + d_V (u)$, for $s(x)+ u \in E.$ It is not hard to see that $d_{(d_V, d_P)} \in \mathrm{Der}_V (E)$. Hence there is a map $t : \mathcal{D}_{\mu, \rho} \rightarrow \mathrm{Der}_V (E)$ given by $t (d_V, d_P) := d_{(d_V, d_P)}$, for $(d_V, d_P) \in \mathcal{D}_{\mu, \rho}$. The map $t$ is indeed a Lie algebra homomorphism and additionally  $\eta t = \mathrm{Id}_{\mathcal{D}_{\mu, \rho}}$. As a result, we get that (\ref{ses3}) is a split short exact sequence of Lie algebras. Thus, we obtain the following isomorphism as Lie algebras
\begin{align*}
    \mathrm{Der}_V (E) = \mathrm{Der}_V (P \ltimes V) \cong \mathcal{D}_{\mu, \rho} \ltimes \mathrm{Der} (P, V).
\end{align*}

\medskip

\section{Proto-twilled Poisson algebras and deformation maps}\label{sec5}
In this section, we consider proto-twilled Poisson algebras and study deformation maps on them. We observe that various well-known operators on Poisson algebras such as Poisson algebra homomorphisms, Poisson derivations, Rota-Baxter operators of any weight, crossed homomorphisms, twisted Rota-Baxter operators, Reynolds operators and modified Rota-Baxter operators can be seen as deformation maps in suitable proto-twilled Poisson algebras. Finally, given a proto-twilled commutative algebra $(\mathcal{P} = P_1 \oplus P_2 , ~ \! \odot ~ \!)$ and a deformation map $r$ on it, we consider scalar deformations of $r$ with respect to an associative algebra deformation of $(\mathcal{P} = P_1 \oplus P_2, ~ \! \odot ~ \! )$. Then the map $r$ turns out to be a deformation map in the semi-classical proto-twilled Poisson algebra.

\begin{definition}
    A {\bf proto-twilled Poisson algebra} is a Poisson algebra $(\mathcal{P}, \odot, \{ \! \! \{ ~, ~ \} \! \! \} )$ whose underlying vector space has a direct sum decomposition $\mathcal{P} = P_1 \oplus P_2$ into subspaces.
\end{definition}

It is important to remark that, in a proto-twilled Poisson algebra  $(\mathcal{P}, \odot, \{ \! \! \{ ~, ~ \} \! \! \} )$ as above, $P_1$ and $P_2$ are just subspaces (need not be Poisson subalgebras) of $\mathcal{P}$. When exactly one of them is a Poisson subalgebra of $\mathcal{P}$, we call $(\mathcal{P}, \odot, \{ \! \! \{ ~, ~ \} \! \! \} )$ a {\bf quasi-twilled Poisson algebra}. Further, it is said to be a {\bf twilled Poisson algebra} when $P_1$ and $P_2$ are both Poisson subalgebras of $\mathcal{P}$. A twilled Poisson algebra is also called a {\em matched pair of Poisson algebras} in the literature \cite{ni-bai}.

\medskip

Let $(\mathcal{P} = P_1 \oplus P_2, \odot, \{ \! \! \{ ~, ~ \} \! \! \})$ be a proto-twilled Poisson algebra. Since $P_1$ and $P_2$ are both subspaces of $\mathcal{P} = P_1 \oplus P_2$, there are $12$ (bi)linear maps
\begin{align}\label{12-maps}
\begin{cases}
    ~ \cdot_1 : P_1 \times P_1 \rightarrow P_1 \quad (\text{symmetric})\\
    ~ \cdot_2 : P_2 \times P_2 \rightarrow P_2  \quad (\text{symmetric})\\
    ~ \mu : P_1 \rightarrow \mathrm{End} (P_2) \\
    ~ \nu : P_2 \rightarrow \mathrm{End} (P_1) \\
    ~ h : P_1 \times P_1 \rightarrow P_2  \quad (\text{symmetric})\\
    ~ \theta : P_2 \times P_2 \rightarrow P_1  \quad (\text{symmetric})
\end{cases}
\qquad \quad
\begin{cases}
    ~ \{ ~, ~ \}_1 : P_1 \times P_1 \rightarrow P_1  \quad (\text{antisymmetric})\\
    ~ \{ ~ , ~ \}_2 : P_2 \times P_2 \rightarrow P_2 \quad (\text{antisymmetric}) \\
    ~ \rho : P_1 \rightarrow \mathrm{End} (P_2) \\
    ~ \psi : P_2 \rightarrow \mathrm{End} (P_1) \\
    ~ H : P_1 \times P_1 \rightarrow P_2 \quad (\text{antisymmetric}) \\
    ~ \Theta : P_2 \times P_2 \rightarrow P_1 \quad (\text{antisymmetric})
\end{cases}
\end{align}
such that the operations $\odot$ and $\{ \! \! \{ ~, ~ \} \! \! \}$ can be expressed as
\begin{align*}
    (x, u) \odot (y, v) =~&  \big( x \cdot_1 y + \nu_u y + \nu_v x + \theta (u, v) ~ \! , ~ \! u \cdot_2 v + \mu_x v + \mu_y u + h (x, y) \big),\\
    \{ \! \! \{ (x, u), (y, v) \} \! \! \} =~& \big( \{ x, y \}_1 + \psi_u y - \psi_v x + \Theta (u, v) ~ \! , ~ \! \{ u, v \}_2 + \rho_x v - \rho_y u + H (x, y) \big),
\end{align*}
for $(x, u), (y, v) \in P_1 \oplus P_2.$ These $12$ (bi)linear operations defined above satisfy certain compatibility conditions to make $(P_1 \oplus P_2, \odot, \{ \! \! \{ ~, ~  \} \! \! \})$ into a Poisson algebra.

There are various examples of proto-twilled Poisson algebras arising from direct products, semidirect products, twisted semidirect products etc. Here we shall provide those examples which are useful in the study of some well-known operators on Poisson algebras.

\begin{exam} (Direct product)
    Let $({P}_1, ~ \! \cdot_1 ~ \!, \{  ~, ~ \}_1 )$ and $({P}_2, ~ \! \cdot_2 ~ \!, \{  ~, ~ \}_2)$ be two Poisson algebras. Then their {\bf direct product} $( P_1 \oplus P_2, ~ \! \cdot_\mathrm{dir} ~ \!, \{ ~, ~ \}_\mathrm{dir})$ is a proto-twilled Poisson algebra, where
    \begin{align*}
        (x, u) \cdot_\mathrm{dir} (y, v) := (x \cdot_1 y ~ \!, ~ \! u \cdot_2 v), \quad \{ (x, u), (y, v) \}_\mathrm{dir} := ( \{ x, y \}_1 ~ \!, ~ \! \{ u, v \}_2), \text{ for } (x, u), (y, v) \in P_1 \oplus P_2.
    \end{align*}
\end{exam}

\begin{exam}\label{semi} (Semidirect product) Let $({P}, ~ \! \cdot ~ \!, \{  ~, ~ \} )$ be a Poisson algebra and $(V, \mu, \rho)$ be a representation of it. Then their {\bf semidirect product} $(P \oplus V, ~ \! \cdot_\ltimes ~ \!, \{ ~, ~ \}_\ltimes )$ is a proto-twilled Poisson algebra, where
\begin{align*}
    (x, u) \cdot_\ltimes (y, v) := ( x \cdot y ~ \! , ~ \! \mu_x v + \mu_y u) ~~~ \text{ and } ~~~ \{ (x, u), (y, v) \}_\ltimes := ( \{ x, y \} ~ \! , ~ \! \rho_x v - \rho_y u ),
\end{align*}
for $(x, u), (y, v) \in P \oplus V$. Symmetrically, $(V \oplus P, ~ \! \cdot_\rtimes ~ \! , \{ ~, ~ \}_\rtimes)$ is also a proto-twilled Poisson algebra, where
\begin{align*}
    (u, x) \cdot_\rtimes (v, y) := (\mu_x v + \mu_y u ~ \! , ~ \! x \cdot y) ~~~~ \text{ and } ~~~~ \{ (u, x) , (v, y) \}_\rtimes := ( \rho_x v - \rho_y u ~ \! , ~ \! \{ x, y \} ).
\end{align*}
    \end{exam}

\medskip

Recall that \cite{ni-bai} an {\bf action} of a Poisson algebra $(P_1 , ~ \! \cdot_1 ~ \!, \{ ~, ~ \}_1)$ on another Poisson algebra $(P_2 , ~ \! \cdot_2 ~ \!, \{ ~, ~ \}_2)$ is given by maps $\mu, \rho : P_1 \rightarrow \mathrm{End}(P_2)$ that makes the vector space $P_2$ into a representation of the Poisson algebra $(P_1, ~\! \cdot_1 ~\!, \{ ~, ~ \}_1)$ such that for all $x \in P_1$ and $u,v \in P_2$,
\begin{align*}
    \mu_x (u \cdot_2 v) =~& \mu_x (u) \cdot_2 v,\\
     \mu_x \{ u, v \}_2 =~&  \{ \mu_x u , v \}_2 - u \cdot_2 \rho_x v, \\
     \rho_x ( u \cdot_2 v) =~& \rho_x (u) \cdot_2 v + u \cdot_2 \rho_x (v),\\
    \rho_x \{ u, v \}_2 =~& \{ \rho_x u, v \}_2 + \{ u, \rho_x v \}_2.
\end{align*}

\begin{exam}\label{semi-new}
    Given an action of a Poisson algebra $(P_1 , ~ \! \cdot_1 ~ \!, \{ ~, ~ \}_1)$ on the Poisson algebra $(P_2 , ~ \! \cdot_2 ~ \!, \{ ~, ~ \}_2)$ as above, the direct sum $P_1 \oplus P_2$ inherits a Poisson algebra structure with the operations
\begin{align*}
    (x, u) \cdot_\ltimes (y, v) := (x \cdot_1 y ~ \! , ~ \!  u \cdot_2 v + \mu_x v + \mu_y u ) ~ \text{ and } ~ \{ (x, u), (y, v ) \}_\ltimes := ( \{ x, y \}_1 ~ \! , ~ \! \{ u, v \}_2 + \rho_x v - \rho_y u ),
\end{align*}
for $(x, u), (y, v) \in P_1 \oplus P_2$. Hence $(P_1 \oplus P_2, ~\! \cdot_\ltimes ~ \! , \{ ~, ~ \}_\ltimes)$ is a proto-twilled Poisson algebra. Symmetrically, $(P_2 \oplus P_1 , ~ \! \cdot_\rtimes ~\! , \{ ~, ~ \}_\rtimes)$ is also a proto-twilled Poisson algebra, where
\begin{align*}
    (u, x) \cdot_\rtimes (v, y) := ( u \cdot_2 v + \mu_x v + \mu_y u  ~ \! , ~ \! x \cdot_1 y) ~~~ \text{ and } ~~~ \{ (u, x), (v, y) \}_\rtimes := (\{ u, v \}_2 + \rho_x v - \rho_y u  ~ \! , ~ \! \{ x, y \}_1 ).
\end{align*}
\end{exam}

 \begin{exam} (Twisted semidirect product) Let $(P, ~ \! \cdot ~ \!, \{ ~, ~ \})$ be a Poisson algebra and $(V, \mu, \rho)$ be a representation of it. Then for any Poisson $2$-cocycle $(h, H) \in Z^2 (P, V)$, the direct sum $P \oplus V$ carries a Poisson algebra structure whose multiplication and the Lie bracket are respectively given by
    \begin{align*}
        (x, u) \cdot_{\ltimes_h} (y, v) := ( x \cdot y ~ \! , ~ \!  \mu_x v + \mu_y u + h (x, y)) ~~ \text{ and } ~~
        \{ (x, u), (y, v) \}_{\ltimes_H} := ( \{ x, y \} ~ \! , ~ \! \rho_x v - \rho_y u + H (x, y)),
    \end{align*}
    for $(x, u), (y, v) \in P \oplus V$. The Poisson algebra $(P \oplus V, ~\! \cdot_{\ltimes_h} ~\! , \{ ~, ~ \}_{\ltimes_H})$ is a proto-twilled Poisson algebra.
         \end{exam}

    \begin{exam} (Reynolds semidirect product) This example is a particular case of the previous one where we consider the adjoint representation of a Poisson algebra and the Poisson $2$-cocycle  is to be the negative of the given Poisson structure. More precisely, let $(P, ~ \! \cdot ~ \!, \{ ~, ~ \})$ be a Poisson algebra. Then $(P \oplus P , ~ \! \cdot_{\ltimes_{-}} ~ \!, \{ ~, ~ \}_{\ltimes_{-}})$ is a proto-twilled Poisson algebra, where for $(x, u), (y, v ) \in P \oplus P$,
    \begin{align*}
        (x, u) \cdot_{\ltimes_{-}} (y, v) :=~& ( x \cdot y ~ \! , ~ \! x \cdot v + u \cdot y - x \cdot y),\\
        \{ (x, u), (y, v ) \}_{\ltimes_{-}} :=~& ( \{ x, y \} ~ \! , ~ \! \{ x, v \} + \{ u, y \} - \{ x, y \} ).
    \end{align*}
         \end{exam}

\medskip

         Another example of a proto-twilled Poisson algebra that is important in our study is given in the following result.

    \begin{proposition}
Let $(P, ~ \! \cdot ~ \!, \{~, ~ \})$ be a Poisson algebra. Then the triple $(P \oplus P, ~ \! \cdot_\mathrm{mod} ~ \!, \{ ~, ~ \}_\mathrm{mod})$ is a proto-twilled Poisson algebra, where
\begin{align*}
    (x, u) \cdot_\mathrm{mod} (y, v) :=~& ( x\cdot y + u \cdot v ~ \! , ~ \! x \cdot v + u \cdot y), \\
    \{ (x, u) , (y, v) \}_\mathrm{mod} :=~& ( \{ x, y \} + \{ u, v \} ~ \! , ~ \! \{ x, v \} + \{ u, y \}),
\end{align*}
for $(x, u), (y, v) \in P \oplus P$. This is called the modified semidirect product.
    \end{proposition}
        
\begin{proof}
    It has been shown in \cite{das-mandal} that the multiplication $~\cdot_\mathrm{mod}~$ is associative. Further, it is commutative as the multiplication $~\cdot~$ is so. On the other hand, it is shown in \cite{sheng-quasi} that the bracket $\{ ~, ~ \}_\mathrm{mod}$ defines a Lie bracket on $P \oplus P$. Thus, it remains to show the Leibniz rule of a Poisson algebra. For any $(x, u), (y, v), (z, w) \in P \oplus P$, we have
    \begin{align*}
        &\{ (x, u) ,  (y, v) \cdot_\mathrm{mod} (z, w) \}_\mathrm{mod} \\
        &=  \{ (x, u), (y \cdot z + v \cdot w ~ \! , ~ \! y \cdot w + v \cdot z) \}_\mathrm{mod} \\
        &= \big(  \{ x, y \cdot z \} + \{ x, v \cdot w \} + \{ u , y \cdot w \} + \{ u , v \cdot z \} ~ \! , ~ \! \{x  , y \cdot w \} + \{ x, v \cdot z \} + \{ u, y \cdot z \} + \{ u , v \cdot w \}   \big) \\
        &= \big(   \{ x, y \} \cdot z + y \cdot \{ x, z \} + \{ x, v \} \cdot w + v \cdot \{ x, w \} + \{ u, y \} \cdot w + y \cdot \{ u, w \} + \{ u, v \} \cdot z + v \cdot \{ u, z \} ,\\
         & \qquad \{ x, y \} \cdot w + y \cdot \{ x, w \} + \{ x, v \} \cdot z + v \cdot \{ x, z \} + \{ u, y \} \cdot z + y \cdot \{ u, z \} + \{ u, v \} \cdot w + v \cdot \{ u, w \} \big) \\
         &= ( \{ x, y \} + \{ u, v \} ~ \! , ~ \! \{ x, v \} + \{ u, y \}) \cdot_\mathrm{mod} (z, w) ~+~ (y, v) \cdot_\mathrm{mod} ( \{x, z \} + \{ u, w \} ~ \! , ~ \! \{ x, w \} + \{ u, z \} ) \\
         &= \{ (x, u), (y, v ) \}_\mathrm{mod} \cdot_\mathrm{mod} (z, w) ~+~ (y, v) \cdot_\mathrm{mod} \{ (x, u), (z, w) \}_\mathrm{mod}.
    \end{align*}
    This proves the desired result.
\end{proof}

\medskip

\noindent {\bf Deformation maps in a proto-twilled Poisson algebra.} The notion of deformation maps in a twilled Poisson algebra was considered by Agore and Militaru \cite{agore} in the classifying compliments problem for Poisson algebras. Here we shall generalize their notion in a given proto-twilled Poisson algebra. Although deformation maps in a proto-twilled Poisson algebra do not solve any compliments problem, they unify a wide class of well-known operators on Poisson algebras. As a result, one could study all these operators by studying deformation maps.

\begin{definition}
    Let $(\mathcal{P} = P_1 \oplus P_2, \odot , \{ \! \! \{ ~, ~\} \! \! \})$ be a proto-twilled Poisson algebra. A linear map $r: P_2 \rightarrow P_1$ is said to be a {\bf deformation map} if the graph $Gr (r) = \{ ( r(u), u) ~ \! |~ \!  u \in P_2 \}$ is a Poisson subalgebra of $(\mathcal{P} , \odot , \{ \! \! \{ ~, ~\} \! \! \})$.
\end{definition}

We have seen that a proto-twilled Poisson algebra $(\mathcal{P} = P_1 \oplus P_2, \odot, \{ \! \! \{ ~, ~\} \! \! \})$ can be described by 12 (bi)linear maps given in (\ref{12-maps}) that satisfy certain compatibility conditions. With such notations, a linear map $r: P_2 \rightarrow P_1$ is a deformation map if and only if for all $u, v \in P_2$,
\begin{align}
    r(u) \cdot_1 r (v) + \nu_{u} r(v) + \nu_v r(u) + \theta (u, v) =~& r \big( u \cdot_2 v + \mu_{r(u)} v + \mu_{r(v)} u + h (r(u), r(v)) \big), \label{dm1}\\
    \{ r(u), r(v) \}_1 + \psi_u r (v) - \psi_v r(u) +  \Theta (u, v)  =~& r \big( \{ u, v \}_2 + \rho_{r(u)} v - \rho_{r(v)} u + H (r(u), r(v)) \big). \label{dm2}
\end{align}

\medskip

\begin{proposition}\label{prop-p2}
    Let $(\mathcal{P} = P_1 \oplus P_2, \odot , \{ \! \! \{ ~, ~\} \! \! \})$ be a proto-twilled Poisson algebra and $r : P_2 \rightarrow P_1$ be a deformation map on it. Then the vector space $P_2$ inherits a Poisson algebra structure whose multiplication and the Lie bracket are respectively given by
\begin{align*}
    u \cdot_r v :=~&  \mathrm{pr}_2 ((r(u), u) \odot (r(v), v)) = u \cdot_2 v + \mu_{r(u)} v + \mu_{r(v)} u + h (r(u), r(v)), \\
    \{ u, v \}_r :=~& \mathrm{pr}_2 \{ \! \! \{ (r(u), u), (r(v), v)\} \! \! \} = \{ u, v \}_2 + \rho_{r(u)} v - \rho_{ r(v)} u + H (r(u), r(v)),
\end{align*}
for $u, v \in P_2$. (We denote this induced Poisson algebra $(P_2, ~ \! \cdot_r ~ \! , \{ ~, ~ \}_r)$ simply by $(P_2)_r$.)
\end{proposition}

In the following, we provide a list of operators (defined on a Poisson algebra) that can be regarded as deformation maps in suitable proto-twilled Poisson algebras.

\medskip

\begin{exam}
    ({\bf Poisson homomorphisms.}) Let $(P_1, ~\! \cdot_1 ~\! , \{ ~, ~\}_1)$ and $(P_2, ~\! \cdot_2 ~\! , \{ ~, ~\}_2)$ be two Poisson algebras. A {\bf Poisson homomorphism} or a homomorphism of Poisson algebras from $P_1$ to $P_2$ is a linear map $ \varphi: P_1 \rightarrow P_2$ that satisfies $ \varphi (x \cdot_1 y) = \varphi(x) \cdot_2 \varphi (y)$ and $\varphi (\{ x, y \}_1) = \{ \varphi(x), \varphi(y) \}_2$, for all $x, y \in P_1.$ Equivalently, $\varphi$ is simultaneously a commutative algebra homomorphism and Lie algebra homomorphism. It is easy to see that a linear map $ \varphi: P_1 \rightarrow P_2$ between two Poisson algebras is a Poisson homomorphism if and only if its graph $Gr (\varphi) = \{ ( \varphi(x), x) ~ \! |~ \! x \in P_1 \}$ is a Poisson subalgebra of the direct product $(P_2 \oplus P_1, ~ \! \cdot_\mathrm{dir} ~\!, \{ ~, ~\}_\mathrm{dir})$. Thus, it turns out that a Poisson homomorphism $\varphi: P_1 \rightarrow P_2$ can be regarded as a deformation map in the proto-twilled Poisson algebra $(P_2 \oplus P_1, ~ \! \cdot_\mathrm{dir} ~\!, \{ ~, ~\}_\mathrm{dir})$.
\end{exam}

\begin{exam}
    ({\bf Poisson derivations.}) Let $(P, ~\! \cdot ~\! , \{ ~, ~\})$ be a Poisson algebra and $(V, \mu, \rho)$ be a representation of it. A linear map $D : P \rightarrow V$ is said to be a {\bf Poisson derivation} on $P$ with values in the representation $V$ if 
\begin{align*}
    D ( x \cdot y) = \mu_x D (y) + \mu_y D(x) \quad  \text{ and } \quad D (\{ x, y \}) = \rho_x D (y) -\rho_y D (x), \text{ for all } x, y \in P.
\end{align*}
A Poisson derivation is nothing but a Poisson $1$-cocycle of the Poisson algebra $P$ with coefficients in the representation $V$. It can be checked that a linear map $D : P \rightarrow V$ is a Poisson derivation if and only if its graph $Gr (D) = \{ (D(x), x) ~ \! | ~ \! x \in P \} \subset V \oplus P$ is a Poisson subalgebra of the semidirect product $(V \oplus P, ~\! \cdot_\rtimes ~\!, \{ ~, ~\}_\rtimes)$ given in Example \ref{semi}. Thus, a Poisson derivation can be seen as a deformation map in the proto-twilled Poisson algebra $(V \oplus P, ~ \! \cdot_\rtimes ~\!, \{ ~, ~ \}_\rtimes)$.
\end{exam}

\begin{exam}
    ({\bf Relative Rota-Baxter operators (of weight $0$).}) Let $(P, ~\! \cdot ~\! , \{ ~, ~\})$ be a Poisson algebra and $(V, \mu, \rho)$ be a representation of it. A linear map $r : V \rightarrow P$ is said to be a {\bf relative Rota-Baxter operator} \cite{ni-bai} (also called an {\bf $\mathcal{O}$-operator}) with respect to the representation $V$ if
\begin{align*}
    r (u) \cdot r( v) = r ( \mu_{r(u)} v + \mu_{r(v)} u) \quad \text{ and } \quad \{ r(u), r(v) \} = r ( \rho_{r(u)} v - \rho_{r(v)} u), \text{ for all } u, v \in V.
\end{align*}
It follows that a linear map $r: V \rightarrow P$ is a relative Rota-Baxter operator if and only if the graph $Gr (r) = \{ (r(u), u) ~ \! | ~ \! u \in V \}$ is a Poisson subalgebra of the semidirect product $(P \oplus V, ~ \! \cdot_\ltimes ~ \!, \{ ~, ~ \}_\ltimes)$. Hence a relative Rota-Baxter operator can also be seen as a deformation map in a proto-twilled Poisson algebra.

When $(V, \mu, \rho) = (P, \mu = \mathrm{ad}_{\cdot } ~ \!, \rho =\mathrm{ad}_{ \{ ~, ~ \} } )$ is the adjoint representation of the Poisson algebra $(P, ~ \! \cdot ~ \!, \{ ~, ~ \})$, a relative Rota-Baxter operator is simply called a Rota-Baxter operator of weight $0$ on the Poisson algebra $P$ \cite{aguiar}. Rota-Baxter operators of weight $0$ play an important role in the study of pre-Poisson algebras and Poisson bialgebras.
\end{exam}

\begin{exam}
    ({\bf Relative Rota-Baxter operators of weight $1$.}) 
Let $(P_1, ~\! \cdot_1 ~\! , \{ ~, ~\}_1)$ be a Poisson algebra that acts on another Poisson algebra $(P_2, ~\! \cdot_2 ~\! , \{ ~, ~\}_2)$ by the maps $\mu, \rho : P_1 \rightarrow \mathrm{End} (P_2).$ Then a linear map $r: P_2 \rightarrow P_1$ is said to be a {\bf relative Rota-Baxter operator of weight $1$} if it satisfies
\begin{align*}
    r (u) \cdot_1 r (v) = r ( u \cdot_2 v + \mu_{r(u)} v + \mu_{r(v)} u ) ~~ \text{ and } ~~ \{ r(u), r(v) \} = r (  \{ u, v \}_2 + \rho_{r(u)} v - \rho_{r(v)} u ),
\end{align*}
    for all $u, v \in P_2$. When the Poisson algebra structure on $P_2$ is trivial, then $r$ becomes a relative Rota-Baxter operator (of weight $0$).

\begin{proposition}
\begin{itemize}
    \item[(i)] For any Poisson algebra $(P, ~\! \cdot ~\! , \{ ~, ~\})$, the map $- \mathrm{Id} : P \rightarrow P$, $x \mapsto - x$ is a  Rota-Baxter operator of weight $1$.
    \item[(ii)] Moreover, if $r : P \rightarrow P$ is a Rota-Baxter operator of weight $1$ then $- \mathrm{Id} - r : P \rightarrow P$ is so.
\end{itemize}
\end{proposition}

The following result shows that relative Rota-Baxter operators of weight $1$ can also be characterized by their graphs.

\begin{proposition}
    Let $(P_1, ~\! \cdot_1 ~\! , \{ ~, ~\}_1)$ be a Poisson algebra that acts on another Poisson algebra $(P_2, ~\! \cdot_2 ~\! , \{ ~, ~\}_2)$ by the maps $\mu, \rho : P_1 \rightarrow \mathrm{End} (P_2).$ A linear map $r : P_2 \rightarrow P_1$ is a relative Rota-Baxter operator of weight $1$ if and only if the graph $Gr (r) = \{ (r(u), u) ~ \! | ~ \! u \in P_2 \}$ is a Poisson subalgebra of $(P_1 \oplus P_2, ~ \! \cdot_\ltimes ~ \! , \{ ~, ~ \}_\ltimes)$ given in Example \ref{semi-new}.
\end{proposition}

It follows from the above result that a relative Rota-Baxter operator of weight $1$ is nothing but a deformation map in the proto-twilled Poisson algebra $(P_1 \oplus P_2, ~ \! \cdot_\ltimes ~ \!, \{ ~, ~ \}_\ltimes)$.
\end{exam}

\begin{exam}
    ({\bf Crossed homomorphisms.}) Let $(P_1, ~\! \cdot_1 ~\! , \{ ~, ~\}_1)$ be a Poisson algebra that acts on another Poisson algebra $(P_2, ~\! \cdot_2 ~\! , \{ ~, ~\}_2)$ by the maps $\mu, \rho : P_1 \rightarrow \mathrm{End} (P_2).$ Then a linear map $D : P_1 \rightarrow P_2$ is said to be a {\bf relative crossed homomorphism} on $P_1$ with values in $P_2$ if
\begin{align*}
    D (x \cdot_1 y) = \mu_x D(y) + \mu_y D (x) + D (x) \cdot_2 D(y) ~~ \text{ and  } ~ ~ D (\{ x, y \}_1) = \rho_x D(y) - \rho_y D(x) + \{ D(x) , D(y)\}_2,
\end{align*}
    for all $x, y \in P_1$. If the Poisson algebra structure on $P_2$ is trivial, then a relative crossed homomorphism becomes a Poisson derivation. It is important to remark that the inverse of a relative Rota-Baxter operator of weight $1$ (when invertible) is a relative crossed homomorphism on $P_1$ with values in $P_2$. A useful characterization of a (relative) crossed homomorphism that is important in our purpose is the following.
\end{exam}

\begin{proposition}
Let $(P_1, ~\! \cdot_1 ~\! , \{ ~, ~\}_1)$ be a Poisson algebra that acts on another Poisson algebra $(P_2, ~\! \cdot_2 ~\! , \{ ~, ~\}_2)$ by the maps $\mu, \rho : P_1 \rightarrow \mathrm{End} (P_2).$ A linear map $D: P_1 \rightarrow P_2$ is a relative crossed homomorphism if and only if its graph $Gr (D) = \{ (D(x), x) ~\! | ~ \! x \in P_1 \}$ is a Poisson subalgebra of the proto-twilled Poisson algebra $(P_2 \oplus P_1, ~\! \cdot_\rtimes ~ \! , \{ ~, ~ \}_\rtimes)$ considered in Example \ref{semi-new}.
\end{proposition}

    Thus, a relative crossed homomorphism can be regarded as a deformation map in $(P_2 \oplus P_1, ~\! \cdot_\rtimes ~ \!, \{ ~, ~ \}_\rtimes)$.
  
\begin{exam}
    ({\bf Twisted Rota-Baxter operators.})
    Let $(P, ~ \! \cdot ~ \!, \{ ~, ~ \})$ be a Poisson algebra and $(V, \mu, \rho)$ be a representation of it. For any Poisson $2$-cocycle $(h, H) \in Z^2 (P, V)$, a linear map $r : V \rightarrow P$ is said to be a {\bf $(h, H)$-twisted Rota-Baxter operator} if it satisfies
    \begin{align*}
        r (u) \cdot r (v) =~& r \big( \mu_{r(u)} v + \mu_{r(v)} u + h (r(u), r(v)) \big),\\
        \{ r(u), r(v) \} =~& r \big(  \rho_{r(u)} v - \rho_{r (v)} u + H ( r(u), r(v))  \big),
    \end{align*}
    for all $u, v \in V$. Twisted Rota-Baxter operators on Poisson algebras were studied and the underlying algebraic structure was found in \cite{das-baishya} (see also \cite{ospel}). An important characterization of a twisted Rota-Baxter operator is given by the following.

    \begin{proposition}
        Let $(P, ~ \! \cdot ~ \!, \{ ~, ~ \})$ be a Poisson algebra, $(V, \mu, \rho)$ be a representation and $(h, H)$ be a Poisson $2$-cocycle. Then a linear map $r : V \rightarrow P$ is a $(h, H)$-twisted Rota-Baxter operator if and only if the graph $Gr (r) = \{   (r (u), u) ~\! | ~ \! u \in V \}$ is a Poisson subalgebra of the twisted semidirect product $(P \oplus V, ~\!  \cdot_{\ltimes_h} ~ \! , \{ ~, ~ \}_{\ltimes_H})$.
    \end{proposition}

    The above characterization enables us to realize a $(h, H)$-twisted Rota-Baxter operator $r : V \rightarrow P$ as a deformation map in the twisted semidirect product $(P \oplus V, ~ \! \cdot_{\ltimes_h} ~\! , \{ ~, ~ \}_{\ltimes_H})$.
\end{exam}

\begin{exam}
    ({\bf Reynolds operators.}) Let $(P, ~ \! \cdot ~ \!, \{ ~, ~ \})$ be a Poisson algebra. A linear map $r : P \rightarrow P$ is said to be a {\bf Reynolds operator} on $P$ if
 \begin{align}
 r (x) \cdot r (y) =~& r \big( r(x) \cdot y + x \cdot r (y) - r(x) \cdot r(y) \big),\label{rey-iden1}\\ 
     \{ r(x), r(y) \} =~& r \big( \{ r(x), y \} + \{ x, r(y) \} - \{ r(x) , r(y) \} \big),\label{rey-iden2}
 \end{align}
 for all $x, y \in P$. Reynolds operators are examples of twisted Rota-Baxter operators. It is important to note that invertible Reynolds operators are closely related to Poisson derivations.

 \begin{proposition}
     Let $(P, ~ \! \cdot ~ \!, \{ ~, ~ \})$ be a Poisson algebra. An invertible linear map $r : P \rightarrow P$ is a Reynolds operator on $P$ if and only if the map $r^{-1} - \mathrm{Id} : P \rightarrow P$ is a Poisson derivation.
 \end{proposition}

 \begin{proof}
     In view of the invertibility of $r$, the identities (\ref{rey-iden1}), (\ref{rey-iden2})  are respectively equivalent to
     \begin{align*}
         (r^{-1} - \mathrm{Id}) (x \cdot y) =~& (r^{-1} - \mathrm{Id}) (x) \cdot y + x \cdot (r^{-1} - \mathrm{Id}) (y),\\
         (r^{-1} - \mathrm{Id}) ( \{ x, y \} ) =~& \{ (r^{-1} - \mathrm{Id}) (x), y \} + \{ x, (r^{-1} - \mathrm{Id})(y) \},
     \end{align*}
     for all $x, y \in P$. This proves the desired result.
 \end{proof}

 Finally, an easy observation shows that a linear map $r : P \rightarrow P$ is a Reynolds operator on $P$ if and only if the graph $Gr (r) = \{ (r(x), x) ~\! | ~ \! x \in P \}$ is a Poisson subalgebra of the Reynolds semidirect product $(P \oplus P, ~\! \cdot_{\ltimes_{-}} ~ \! , \{ ~, ~ \}_{\ltimes_{-}})$. Hence Reynolds operators on a Poisson algebra $P$ are equivalent to deformation maps in the proto-twilled Poisson algebra $(P \oplus P, ~\! \cdot_{\ltimes_{-}} ~ \! , \{ ~, ~ \}_{\ltimes_{-}})$.

\end{exam}  

\begin{exam}
({\bf Modified Rota-Baxter operators.}) For a Poisson algebra $(P, ~ \! \cdot ~ \!, \{ ~, ~ \})$, a linear map $r : P \rightarrow P$ is said to be a {\bf modified Rota-Baxter operator} on $P$ if
\begin{align}
    r (x) \cdot r (y) = r ( r(x) \cdot y + x \cdot r (y)) - x \cdot y ~~~ \text{ and } ~~~ \{ r(x), r(y) \} = r ( \{ r(x), y \} + \{ x, r(y) \}) - \{ x , y \},
\end{align}
for all $x, y \in P$. Modified Rota-Baxter operators are closely related to Rota-Baxter operators of weight $1$ by the following result.
\begin{proposition}
    Let  $(P, ~ \! \cdot ~ \!, \{ ~, ~ \})$ be a Poisson algebra. Then a linear map $r : P \rightarrow P$ is a Rota-Baxter operator of weight $1$ if and only if the map $\mathrm{Id} + 2 r : P \rightarrow P$ is a modified Rota-Baxter operator on $P$.
\end{proposition}

\begin{proof}
    By a direct calculation, we have
    \begin{align*}
        &(\mathrm{Id} + 2 r ) (x) \cdot (\mathrm{Id} + 2 r ) (y) - (\mathrm{Id} + 2 r ) \big(  (\mathrm{Id} + 2 r )(x) \cdot y + x \cdot (\mathrm{Id} + 2 r )(y)   \big) + x \cdot y \\
        &= x\cdot y + 2 x \cdot r(y) + 2 r(x) \cdot y + 4 r(x) \cdot r(y) - 2 (\mathrm{Id} + 2 r ) (x \cdot y + r(x) \cdot y + x \cdot r(y)) + x \cdot y \\
        &= 4 \big(   r(x) \cdot r(y) - r ( x \cdot y + r(x) \cdot y + x \cdot r(y))   \big)
    \end{align*}
    and similarly,
    \begin{align*}
        &\{ (\mathrm{Id} + 2 r ) (x) , (\mathrm{Id} + 2 r ) (y) \} - (\mathrm{Id} + 2 r ) \big(  \{ (\mathrm{Id} + 2 r ) (x), y \} + \{ x, (\mathrm{Id} + 2 r ) (y) \}  \big) + \{ x, y \}\\
        &= 4 \big(   \{ r (x) , r(y) \} - r ( \{ x, y \} + \{ r(x), y \} + \{ x, r(y) \} ) \big).
    \end{align*}
    This shows the desired result.
\end{proof}
    
Finally, one can show the following characterization of a modified Rota-Baxter operator on a Poisson algebra.

\begin{proposition}
    Let  $(P, ~ \! \cdot ~ \!, \{ ~, ~ \})$ be a Poisson algebra. Then a linear map $r : P \rightarrow P$ is a modified Rota-Baxter operator on $P$ if and only if $Gr (r) = \{ (r(x), x) ~\! | ~ \! x \in P \}$ is a Poisson subalgebra of the modified semidirect product $(P \oplus P, ~ \! \cdot_\mathrm{mod}, \{ ~, ~ \}_\mathrm{mod})$.
\end{proposition}

The proof of the above proposition is straightforward. As a consequence of this result, we get that a modified Rota-Baxter operator $r : P \rightarrow P$ on a Poisson algebra $P$ is equivalent to a deformation map in the proto-twilled Poisson algebra $(P \oplus P, ~ \! \cdot_\mathrm{mod}, \{ ~, ~ \}_\mathrm{mod})$.
\end{exam}

\subsection{Scalar deformations of $r$} It is well-known that the semi-classical limit of an associative algebra deformation of a commutative algebra carries a Poisson algebra structure \cite{kont}. Recently, this result was extended when the algebras are endowed with representations \cite{chen-bai-guo}. Among others, the authors considered scalar deformations of a Rota-Baxter operator $r$ defined on a given commutative algebra and showed that the same operator $r$ becomes a Rota-Baxter operator on the semi-classical limit Poisson algebra. Here we shall generalize these results by considering proto-twilled associative algebra deformations of a proto-twilled commutative algebra.

\medskip

A {\bf proto-twilled associative algebra} is an associative algebra whose underlying space has a direct sum decomposition into subspaces. Further, it is said to be a {\bf proto-twilled commutative algebra} if the underlying associative algebra is commutative.
Quasi-twilled and twilled associative algebras or commutative algebras can be defined similarly. Let $(\mathcal{P} = P_1 \oplus P_2, \odot)$ be a proto-twilled associative algebra or a proto-twilled commutative algebra. Then a linear map $r: P_2 \rightarrow P_1$ is said to be a {\bf deformation map} if its graph $Gr (r) = \{ ( r(u), u) ~ \! |~ \!  u \in P_2 \}$ is a subalgebra of $(\mathcal{P} = P_1 \oplus P_2, \odot)$.

\medskip

Let  $(\mathcal{P} = P_1 \oplus P_2, \odot)$ be a proto-twilled commutative algebra. Consider the space $\mathcal{P} [ \! \! [ t ] \! \! ] $ of all formal power series in $t$ with coefficients from $\mathcal{P}$. Then we have $\mathcal{P}[ \! \! [t ] \! \! ] \cong P_1 [ \! \! [t ] \! \! ] \oplus P_2 [ \! \! [t ] \! \! ]$ as ${\bf k}[ \! \! [t ] \! \! ]$-modules.
An associative algebra deformation of $(\mathcal{P} = P_1 \oplus P_2, \odot)$ is a formal sum
\begin{align*}
    \odot_t = \odot_0 + t \odot_1 + t^2 \odot_2 + \cdots \in \mathrm{Hom} (\mathcal{P}^{\otimes 2}, \mathcal{P}) [ \! \! [ t] \! \! ] ~ \text{ with } \odot_0 = \odot
\end{align*}
that makes $(\mathcal{P} [ \! \! [t ] \! \! ] \cong P_1 [ \! \! [t ] \! \! ] \oplus P_2 [ \! \! [t ] \! \! ], \odot_t)$ into a proto-twilled associative algebra over the base ring ${\bf k}[ \! \! [t ] \! \! ].$
Then by adapting the semi-classical limit construction, we get the following.

\begin{theorem}
    Let $(\mathcal{P} = P_1 \oplus P_2, \odot)$ be a proto-twilled commutative algebra and let $\odot_t= \sum_{i=0}^\infty t^i \odot_i$ (with $\odot_0 = \odot$) be an associative algebra deformation of it. Define a bracket 
    \begin{align*}
         \{ \! \! \{ ~, ~ \} \! \! \}: \mathcal{P} \times \mathcal{P} \rightarrow \mathcal{P} ~ \text{ by } ~ \{ \! \! \{ (x, u), (y, v) \} \! \! \}:= (x, u) \odot_1 (y, v) - (y, v) \odot_1 (x, u),
    \end{align*}
    for $(x, u) , (y, v) \in \mathcal{P}$. Then $(\mathcal{P} = P_1 \oplus P_2 , \odot, \{ \! \! \{ ~, ~ \} \! \! \})$ is a proto-twilled Poisson algebra.

    In particular, when $(\mathcal{P} = P_1 \oplus P_2, \odot)$ is a quasi-twilled (resp. twilled) commutative algebra and $\odot_t = \sum_{i=0}^\infty t^i \odot_i$ (with $\odot_0 = \odot$) is an associative deformation that makes $(\mathcal{P} [ \! \! [t ] \! \! ] \cong  P_1 [ \! \! [t ] \! \! ] \oplus P_2 [ \! \! [t ] \! \! ], \odot_t )$ into a quasi-twilled (resp. twilled) associative algebra, then $(\mathcal{P} = P_1 \oplus P_2, \odot, \{ \! \! \{ ~, ~ \} \! \! \})$ is a quasi-twilled (resp. twilled) Poisson algebra.
\end{theorem}

\begin{theorem}
    Let $(\mathcal{P} = P_1 \oplus P_2, \odot)$ be a proto-twilled commutative algebra and $r: P_2 \rightarrow P_1$ be a deformation map on it. Suppose there is an associative algebra deformation $\odot_t = \sum_{i=0}^\infty t^i \odot_i$ (with $\odot_0 = \odot$) of the given proto-twilled commutative algebra for which the ${\bf k}[ \! \! [t ] \! \! ]$-linear extension of $r$, namely,
    \begin{align*}
        r_{\bf k} : P_2 [ \! \! [t ] \! \! ] \rightarrow P_1 [ \! \! [t ] \! \! ] ~~~  \text{ defined by } ~~~ r_{\bf k} ( \sum_{i=0}^\infty u_i t^i) = \sum_{i=0}^\infty r(u_i) t^i
    \end{align*}
    is a deformation map in the proto-twilled associative algebra $(\mathcal{P} [ \! \! [t ] \! \! ] \cong P_1 [ \! \! [t ] \! \! ] \oplus P_2 [ \! \! [t ] \! \! ], \odot_t)$. Then the linear map $r $ is a deformation map in the semi-classical proto-twilled Poisson algebra $(\mathcal{P} = P_1 \oplus P_2, \odot, \{ \! \! \{ ~, ~ \} \! \! \}).$
\end{theorem}

\begin{proof}
   For any $u, v \in P_2$, consider the elements $(r(u), u)$ and $(r(v), v)$ from $Gr (r)$. Realizing these elements in $Gr (r_{\mathbf{k}})$ and using the fact that $Gr (r_{\mathbf{k}})$  is a subalgebra of $(\mathcal{P} [ \! \! [t ] \! \! ], \odot_t)$, we get that
   \begin{align*}
        (r(u) , u) \odot_t (r(v), v) -  (r(v) , v) \odot_t (r(u), u) \in Gr ( r_{\mathbf{k}}).
   \end{align*}
   Diving this element by $t$ and then taking $t$ tends to $0$, we get that   $ \{ \! \! \{  (r(u), u),  (r(v), v) \} \! \! \} \in Gr (r)$. This shows that $Gr (r)$ is also a subalgebra of the Lie algebra $(\mathcal{P} = P_1 \oplus P_2, \{ \! \! \{ ~, ~ \} \! \! \})$. Hence the result follows.
\end{proof}

As a consequence of the above theorem, we obtain the following results.

\begin{corollary}
(i) Let $A, B$ be two commutative algebras and $\varphi : A \rightarrow B$ be an algebra homomorphism. Suppose there are associative algebra deformations $(A [ \! \! [ t ] \! \! ], \mu_t)$ and  $(B [ \! \! [ t ] \! \! ], \nu_t)$ of the commutative algebras $A$ and $B$, respectively. If the ${\bf k} [ \! \! [ t ] \! \! ]$-linear map $\varphi_{\bf k} : A [ \! \! [ t ] \! \! ] \rightarrow B [ \! \! [ t ] \! \! ]$ given by $\varphi_{\bf k} ( \sum_{i=0}^\infty a_i t^i ) = \sum_{i=0}^\infty \varphi (a_i) t^i $ is an associative algebra homomorphism from $(A [ \! \! [ t ] \! \! ], \mu_t )$ to $(B [ \! \! [ t ] \! \! ], \nu_t )$ then the map $\varphi : A \rightarrow B$ is a homomorphism between the corresponding semi-classical Poisson algebras.

\medskip

\noindent (ii) Let $A$ be a commutative algebra and $M$ be a module over it. Suppose there is an associative algebra deformation $(A [ \! \! [ t] \! \! ], \mu_t)$ of the commutative algebra $A$ and an $A [ \! \! [ t ] \! \! ]$-bimodule deformation $(M [ \! \! [ t ] \! \! ], l_t, r_t)$ of the $A$-module $M$.

    \medskip

\quad    (a) Let $D : A \rightarrow M$ be a derivation on the commutative algebra $A$ with values in the $A$-module $M$ such that its ${\bf k} [ \! \! [ t ] \! \! ]$-linear extension $D_{\bf k} : A [ \! \! [ t ] \! \! ] \rightarrow M [ \! \! [ t ] \! \! ]$ is a derivation on the deformed algebra $(A  [ \! \! [ t ] \! \! ], \mu_t )$ with values in the $A [ \! \! [ t ] \! \! ]$-bimodule $(M [ \! \! [ t ] \! \! ], l_t, r_t)$. Then the map $D: A \rightarrow M$ is a Poisson derivation on the semi-classical Poisson algebra on $A$ with values in the semi-classical representation on $M$.

    \medskip

  \quad   (b) Let $r: M \rightarrow A$ be a relative Rota-Baxter operator (of weight $0$) such that its ${\bf k} [ \! \! [ t ] \! \! ]$-linear extension $r_{\bf k}: M [ \! \! [ t ] \! \! ] \rightarrow A [ \! \! [ t ] \! \! ]$ is also a relative Rota-Baxter operator on the deformed algebra. Then the map $r: M \rightarrow A$ is a relative Rota-Baxter operator (of weight $0$) on the semi-classical Poisson algebra $A$ with respect to the semi-classical representation on $M$ \cite{chen-bai-guo}.

    \medskip

 \noindent (iii) Let $A, B$ be two commutative algebras and let the algebra $A$ acts on the algebra $B$. Suppose there are associative algebra deformations $(A [ \! \! [ t] \! \! ], \mu_t)$ and $(B [ \! \! [ t] \! \! ], \nu_t)$ of the commutative algebras $A$ and $B$, respectively, and there is an $A [ \! \! [ t ] \! \! ]$-bimodule deformation $(B[ \! \! [ t ] \! \! ])$ that defines an action of the associative algebra $(A [ \! \! [ t] \! \! ], \mu_t)$ on the algebra $(B [ \! \! [ t] \! \! ], \nu_t)$.

    \medskip

  \quad  (a) Let $r : B \rightarrow A$ be a relative Rota-Baxter operator of weight $1$ such that its ${\bf k}[ \! \! [ t ] \! \! ]$-linear extension $r_{\bf k} : B [ \! \! [ t ] \! \! ] \rightarrow A [ \! \! [ t ] \! \! ]$ is also a relative Rota-Baxter operator of weight $1$ on the deformed algebra. Then $r: B \rightarrow A$ is a relative Rota-Baxter operator of weight $1$ on the semi-classical Poisson algebra \cite{chen-bai-guo}.

    \medskip

  \quad  (b) Let $D : A \rightarrow B$ be a crossed homomorphism on $A$ with values in $B$ such that its ${\bf k}[ \! \! [ t ] \! \! ]$-linear extension $D_{\bf k} : A [ \! \! [ t ] \! \! ] \rightarrow B [ \! \! [ t ] \! \! ]$ is a crossed homomorphism on the deformed algebra $(A [ \! \! [ t ] \! \! ], \mu_t)$ with values in $B[ \! \! [ t ] \! \! ]$. Then the map $D: A \rightarrow B$ is a crossed homomorphism on the semi-classical Poisson algebra $A$ with values in the semi-classical Poisson algebra $B$.

    \medskip

  \noindent  (iv) Let $A$ be a commutative algebra and $r: A \rightarrow A$ be a Reynolds operator (resp. modified Rota-Baxter operator) on $A$. Suppose there is an associative algebra deformation $(A [ \! \! [t ] \! \! ], \mu_t)$ of the commutative algebra $A$ for which the ${\bf k} [ \! \! [t ] \! \! ]$-linear extension $ r_{\bf k}: A [ \! \! [ t ] \! \! ] \rightarrow A [ \! \! [ t ] \! \! ]$ is also a Reynolds operator (resp. modified Rota-Baxter operator) on the deformed associative algebra $(A [ \! \! [ t ] \! \! ], \mu_t )$. Then the map $r: A \rightarrow A$ is a Reynolds operator (resp. modified Rota-Baxter operator) on the corresponding semi-classical Poisson algebra.
\end{corollary}

\medskip

\section{Cohomology of deformation maps: Examples}\label{sec6}

This section introduces the cohomology of a deformation map $r$ in a given proto-twilled Poisson algebra. In particular, we obtain the cohomology of Poisson homomorphisms, Poisson derivations, relative Rota-Baxter operators, crossed homomorphisms, twisted Rota-Baxter operators, Reynolds operators and modified Rota-Baxter operators on Poisson algebras. We will use the cohomology of a deformation map $r$ in the next section when we study the formal deformation theory of the operator $r$.

\medskip

Let $(\mathcal{P} = P_1 \oplus P_2, \odot , \{ \! \! \{ ~, ~\} \! \! \})$ be a proto-twilled Poisson algebra. For any linear map $r : P_2 \rightarrow P_1$, consider two new linear maps $\mathrm{Id} + \widetilde{r} : P_1 \oplus P_2 \rightarrow P_1 \oplus P_2$ and $\mathrm{Id} - \widetilde{r} : P_1 \oplus P_2 \rightarrow P_1 \oplus P_2$ defined by
\begin{align*}
    (\mathrm{Id} + \widetilde{r}) (x, u) := (x + r(u), u) \quad \text{ and } \quad (\mathrm{Id} - \widetilde{r}) (x, u) := (x - r(u), u),
\end{align*}
for $(x, u) \in P_1 \oplus P_2$. Then $\mathrm{Id} + \widetilde{r}$ is a linear isomorphism with the inverse $\mathrm{Id} - \widetilde{r}$. Using these maps, one may define a new (proto-twilled) Poisson algebra structure $(\mathcal{P} = P_1 \oplus P_2, \odot_r, \{ \! \! \{ ~, ~\} \! \! \}_r)$ given by
\begin{align*}
    (x, u) \odot_r (y, v) :=~& (\mathrm{Id} - \widetilde{r}) \big( ( \mathrm{Id} + \widetilde{r} ) (x, u) \odot  ( \mathrm{Id} + \widetilde{r} ) (y, v) \big), \\
    \{ \! \! \{ (x, u), (y, v) \} \! \! \}_r :=~& (\mathrm{Id} - \widetilde{r}) \{ \! \! \{ ( \mathrm{Id} + \widetilde{r} ) (x, u) , ( \mathrm{Id} + \widetilde{r} ) (y, v) \} \! \! \},
\end{align*}
for $(x, u), (y, v) \in P_1 \oplus P_2$. Moreover, the map $\mathrm{Id} + \widetilde{r} : P_1 \oplus P_2 \rightarrow P_1 \oplus P_2$ is an isomorphism of Poisson algebras from $(P_1 \oplus P_2, \odot_r, \{ \! \! \{ ~,~\} \! \! \}_r)$ to $(P_1 \oplus P_2, \odot, \{ \! \! \{ ~,~\} \! \! \})$.

\begin{proposition}
Let $(P_1 \oplus P_2, \odot, \{ \! \! \{ ~,~\} \! \! \})$ be a proto-twilled Poisson algebra and $r : P_2 \rightarrow P_1$ be a deformation map on it. Then $(P_1 \oplus P_2, \odot_r, \{ \! \! \{ ~,~\} \! \! \}_r)$ is quasi-twilled Poisson algebra with $P_2$ being a Poisson subalgebra of it.
\end{proposition}

\begin{proof}
    To show that $P_2$ is a Poisson subalgebra of $(P_1 \oplus P_2, \odot_r, \{ \! \! \{ ~,~\} \! \! \}_r)$, we observe that
    \begin{align}
        &(0, u) \odot_r (0, v) \nonumber \\
        &= (\mathrm{Id} - \widetilde{r}) \big( (r(u), u) \odot (r(v), v)  \big) \nonumber \\
        &= (\mathrm{Id} - \widetilde{r}) \big(    r(u) \cdot_1 r (v) + \nu_{u} r(v) + \nu_v r(u) + \theta (u, v) ~ \!, ~ \!  u \cdot_2 v + \mu_{r(u)} v + \mu_{r(v)} u + h (r(u), r(v))      \big)  \nonumber  \\
       &\stackrel{(\ref{dm1})}{=} \big(  0 ~\! , ~\! u \cdot_2 v + \mu_{r(u)} v + \mu_{r(v)} u + h (r(u), r(v))    \big) \in \{ 0 \} \oplus P_2 \label{comm-sub}
    \end{align}
    and similarly,
    \begin{align}
        &\{ \! \! \{  (0, u), (0, v) \} \! \! \}_r \nonumber \\
        &=   (\mathrm{Id} - \widetilde{r}) \big(  \{ r (u), r(v) \}_1 +\psi_u r(v) - \psi_v r(u) + \Theta (u, v) ~ \!, ~ \! \{ u, v\}_2 + \rho_{r (u)} v - \rho_{r (v)} u + H ( r (u), r(v))  \big) \nonumber \\
        &\stackrel{(\ref{dm2})}{=} \big( 0,  \{ u, v\}_2 + \rho_{r (u)} v - \rho_{r (v)} u + H ( r (u), r(v))    \big) \in \{ 0 \} \oplus P_2. \label{lie-sub}
    \end{align}
    This proves the result.
\end{proof}

It follows from (\ref{comm-sub}) and (\ref{lie-sub}) that the Poisson algebra structure on $P_2$ induced from $(P_1 \oplus P_2, \odot_r, \{ \! \! \{ ~,~\} \! \! \}_r)$ coincides with the one given in Proposition \ref{prop-p2}.

\begin{proposition}
    Let $(P_1 \oplus P_2, \odot, \{ \! \! \{ ~,~\} \! \! \})$ be a proto-twilled Poisson algebra and $r : P_2 \rightarrow P_1$ be a deformation map on it. Define maps $\nu_r , \psi_r : P_2 \rightarrow \mathrm{End} (P_1)$ by 
    \begin{align*}
        (\nu_r)_u x :=~& \nu_u x + r(u) \cdot_1 x - r (   \mu_x u + h ( r(u), x) ),\\
        (\psi_r)_u x :=~& \psi_u x + \{ r (u) , x \}_1 - r ( - \rho_x u + H (r(u), x) ),
    \end{align*}
    for $u \in P_2$ and $x \in P_1$. Then the triple $(P_1, \nu_r, \psi_r)$ is a representation of the induced Poisson algebra $(P_2)_r$.
\end{proposition}

\begin{proof}
    We have already seen that $( P_1 \oplus P_2, \odot_r, \{ \! \! \{ ~, ~ \} \! \! \}_r)$ is a quasi-twilled Poisson algebra with $P_2$ being a Poisson subalgebra of it (and the induced Poisson algebra structure being $(P_2)_r$). Moreover, it follows that the Poisson algebra $(P_2)_r$ has a representation on the space $P_1$ with the action maps $\overline{\nu}, \overline{\psi} : P_2 \rightarrow \mathrm{End} (P_1)$ given by
    \begin{align*}
     \overline{\nu}_u x :=    \mathrm{pr}_1 \big( (0, u) \odot_r (x, 0) \big) =~& \mathrm{pr}_1 \big(  (\mathrm{Id} - \widetilde{r}) (  (r(u), u) \odot (x, 0)  )   \big) \\
        =~&  \mathrm{pr}_1 \big(  (\mathrm{Id} - \widetilde{r}) (  r(u) \cdot_1 x + \nu_u x ~ \! , ~ \! \mu_x u + h (r(u), x)  ) \big) \\
        =~& \mathrm{pr}_1 \big(   r(u) \cdot_1 x + \nu_u x - r (\mu_x u + h ( r(u), x) ) ~ \! , ~ \! \mu_x u + h (r(u), x)   \big) \\
        =~& r(u) \cdot_1 x + \nu_u  x - r (\mu_x u + h (r(u), x) ) = (\nu_r)_u x
    \end{align*}
    and
    \begin{align*}
      \overline{\psi}_u x :=  \mathrm{pr}_1 \{ \! \! \{ (0, u), (x, 0) \} \! \! \}_r =~& \mathrm{pr}_1 \big(  (\mathrm{Id} - \widetilde{r}) (  \{ \! \! \{   (r(u), u) , (x, 0)   \} \! \! \} )  \big) \\
        =~& \mathrm{pr}_1 \big(  (\mathrm{Id} - \widetilde{r}) (  \{ r(u), x \}_1 + \psi_u x ~ \! , ~ \! - \rho_x u + H ( r(u) , x)  )  \big) \\
        =~& \mathrm{pr}_1 \big(  \{ r(u), x \}_1 + \psi_u x - r ( -\rho_x u + H ( r(u), x) ) ~ \! , ~ \! -\rho_x u + H (r (u), x)   \big) \\
        =~&  \{ r(u), x \}_1 + \psi_u x - r ( -\rho_x u + H ( r(u), x) ) = (\psi_r)_u x,
    \end{align*}
    for $u \in P_2$ and $x \in P_1$. This proves the result.
\end{proof}

Let $(\mathcal{P} = P_1 \oplus P_2 , \odot, \{ \! \! \{ ~, ~ \} \! \! \})$ be a proto-twilled Poisson algebra and $r : P_2 \rightarrow P_1$ be a deformation map on it. Then one may consider the FGV cochain complex of the induced Poisson algebra $(P_2)_r$ with coefficients in the representation $(P_1, \nu_r, \psi_r)$ given in the above proposition. More precisely, we set
\begin{align*}
    C^k ((P_2)_r, P_1 ) = \oplus_{\substack{m+n= k \\ m \neq 1}} C^{m, n} ((P_2)_r, P_1), \text{ where } C^{m, n} ((P_2)_r, P_1) \text{ is the space of } (m,n)\text{-cochains and}
\end{align*}
maps $\delta^{m,n}_\mathrm{H} : C^{m,n} ((P_2)_r, P_1)  \rightarrow C^{m+1, n} ((P_2)_r, P_1)$ ($m \neq 0$), $ \delta^{0,n}_\mathrm{H} : C^{0,n} ((P_2)_r, P_1) \rightarrow C^{2,n-1} ((P_2)_r, P_1)$ and $\delta_\mathrm{CE}^{m,n} : C^{m,n} ((P_2)_r, P_1)  \rightarrow C^{m, n+1} ((P_2)_r, P_1) $ by
\begin{align*}
   ( \delta^{m,n}_\mathrm{H} f) ((u_1 \otimes \cdots \otimes u_{m+1}) \otimes ~ & ( v_1 \wedge \cdots \wedge v_n ) ) 
   = (\nu_r)_{u_1} f ((u_2 \otimes \cdots \otimes u_{m+1}) \otimes ( v_1 \wedge \cdots \wedge v_n ) ) \\
   &+ \sum_{i=1}^m (-1)^i ~ \! f ((u_1 \otimes \cdots \otimes u_i \cdot_r u_{i+1} \otimes \cdots \otimes u_{m+1}) \otimes ( v_1 \wedge \cdots \wedge v_n ) ) \\
   &+ (-1)^{m+1} ~ \! (\nu_r)_{u_{m+1}} f ((u_1 \otimes \cdots \otimes u_{m}) \otimes ( v_1 \wedge \cdots \wedge v_n ) ),
\end{align*}
\begin{align*}
  \delta_\mathrm{H}^{0, n} :  C^{0, n} ((P_2)_r, P_1)  \hookrightarrow \mathrm{Hom} (P_2 \otimes \wedge^{n-1} P_2, P_1) = C^{1, n-1} ( (P_2)_r, P_1) \xrightarrow{ \delta_\mathrm{H}^{1, n-1}} C^{2, n-1} ((P_2)_r, P_1),
\end{align*}
and
\begin{align*}
    (\delta^{m,n}_\mathrm{CE} f) & ((u_1 \otimes \cdots \otimes u_{m}) \otimes ( v_1 \wedge \cdots \wedge v_{n+1} ) ) \\
    &= \sum_{i=1}^{n+1} (-1)^{i+1} ~ \! \bigg(  (\psi_r)_{v_i} f ((u_1 \otimes \cdots \otimes u_{m}) \otimes ( v_1 \wedge \cdots \wedge \widehat{v_i} \wedge \cdots \wedge v_{n+1} ) ) \\
    & ~ ~ - \sum_{j=1}^m f (( u_1 \otimes \otimes \{ v_i, u_j \}_r \otimes \cdots \otimes u_m ) \otimes (  v_1 \wedge \cdots \wedge \widehat{v_i} \wedge \cdots \wedge v_{n+1}) \bigg) \\
    & ~ ~ + \sum_{1 \leq i < j \leq n+1} (-1)^{i+j} ~ \! f ((u_1 \otimes \cdots \otimes u_{m}) \otimes (\{ v_i, v_j\}_r \wedge v_1 \wedge \cdots \wedge \widehat{v_i} \wedge \cdots \wedge \widehat{v_j} \wedge \cdots \wedge v_{n+1}) ),
\end{align*}
for any $f \in C^{m, n} ((P_2)_r, P_1)$ and $u_1, \ldots, u_{m+1}, v_1, \ldots, v_{n+1} \in P_2$. Finally, we define a map
\begin{align*}
   \delta^r_\mathrm{FGV} : C^k ((P_2)_r, P_1) \rightarrow C^{k+1} ((P_2)_r, P_1) ~~ \text{ by } ~~ \delta_\mathrm{FGV}^r (f) = \delta_\mathrm{H}^{m,n} (f) + (-1)^m ~ \! \delta^{m,n}_\mathrm{CE} (f),
\end{align*}
for $ f \in C^{m,n} ((P_2)_r, P_1).$
Then it turns out that $\{ C^\bullet ( (P_2)_r, P_1), \delta^r_\mathrm{FGV} \}$ is a cochain complex, called the cochain complex associated with the deformation map $r$. The corresponding cohomology groups are called the cohomology groups of the operator $r$, and they are denoted by $H^\bullet_\mathrm{FGV}  ((P_2)_r, P_1)$.

\medskip

 So far, we have defined the cohomology theory of a deformation map in a given proto-twilled Poisson algebra. In various examples of proto-twilled Poisson algebras and deformation maps there, we obtain the cohomologies of Poisson algebra homomorphisms, Poisson derivations, Rota-Baxter operators of weight $0$ and $1$, crossed homomorphisms, twisted Rota-Baxter operators, Reynolds operators and modified Rota-Baxter operators on a Poisson algebra. Thus, the cohomology for any of these operators is described by FGV cohomology of the corresponding induced Poisson algebra with coefficients in a suitable representation.

\medskip

\medskip

\noindent ({\bf Cohomology of Poisson homomorphisms}.) Let $(P_1, ~ \! \cdot_1 ~ \!, \{~, ~ \}_1)$ and $(P_2, ~ \! \cdot_2 ~ \!, \{~, ~ \}_2)$ be two Poisson algebras, and $\varphi : P_1 \rightarrow P_2$ be a Poisson homomorphism. Then the Poisson algebra  $(P_1, ~ \! \cdot_1 ~ \!, \{~, ~ \}_1)$ has a representation on the vector space $P_2$ with the maps $\mu_\varphi, \rho_\varphi : P_1 \rightarrow \mathrm{End} (P_2)$ which are respectively given by
\begin{align*}
    (\mu_\varphi )_x u : = \varphi (x) \cdot_2 u ~~~ \text{ and } (\rho_\varphi)_x u := \{ \varphi (x) , u \}_2, \text{ for } x \in P_1, u \in P_2.
\end{align*}
The FGV cohomology of the Poisson algebra $(P_1,  ~ \! \cdot_1 ~ \! , \{~, ~ \}_1)$ with coefficients in the above representation on $P_2$ is defined to be the cohomology of the Poisson homomorphism $\varphi$.

\medskip

\noindent ({\bf Cohomology of Poisson derivations.}) Let $(P, ~ \! \cdot ~ \! , \{ ~, ~ \})$ be a Poisson algebra and $(V, \mu, \rho)$ be a representation of it. For any Poisson derivation $D : P \rightarrow V$, the FGV cohomology of the given Poisson algebra $(P, ~ \! \cdot ~ \! , \{ ~, ~ \})$ with coefficients in the representation $(V, \mu, \rho)$ is defined to be the cohomology of the Poisson derivation $D$. It follows that the cohomology of a Poisson derivation $D$ is independent of $D$ itself.

\medskip

\noindent ({\bf Cohomology of relative Rota-Baxter operators of weight $0$.}) Let $(P, ~ \! \cdot ~ \! , \{ ~, ~ \})$ be a Poisson algebra, $(V, \mu, \rho)$ be a representation of it and $r: V \rightarrow P$ be a relative Rota-Baxter operator of weight $0$. Then $V$ inherits a Poisson algebra structure whose multiplication and the bracket are respectively given by
\begin{align*}
    u \cdot_r v := \mu_{r(u)} v + \mu_{r(v)} u ~~~~ \text{ and } ~~~~ \{ u, v \}_r := \rho_{r(u)} v - \rho_{r(v)} u, \text{ for } u , v \in V.
\end{align*}
Moreover, there is a representation of this induced Poisson algebra $(V, ~ \! \cdot_r ~ \!, \{ ~, ~ \}_r )$ on the space $P$ with the action maps $\nu_r , \psi_r : V \rightarrow \mathrm{End} (P)$ which are respectively given by
\begin{align*}
    (\nu_r)_u x = r(u) \cdot x - r (\mu_x u) ~~~ \text{ and } ~~~ (\psi_r)_u x = \{ r(u), x \} + r (\rho_x u), \text{ for } u \in V, x \in P.
\end{align*}
The FGV cohomology of the Poisson algebra $(V, ~ \! \cdot_r ~ \!, \{ ~, ~ \}_r )$ with coefficients in the representation $(P, \nu_r, \psi_r)$ is called the cohomology of the relative Rota-Baxter operator $r$ (of weight $0$).

\medskip

\noindent ({\bf Cohomology of relative Rota-Baxter operators of weight $1$.}) Let $(P_1, ~ \! \cdot_1 ~ \! , \{ ~, ~ \}_1)$ be a Poisson algebra that acts on another Poisson algebra $(P_2, ~ \! \cdot_2 ~ \! , \{ ~, ~ \}_2)$ by the maps $\mu, \rho : P_1 \rightarrow \mathrm{End} (P_2)$. Suppose $r: P_2 \rightarrow P_1$ is a relative Rota-Baxter operator of weight $1$. Then the space $P_2$ carries a new Poisson algebra structure whose multiplication and the bracket are respectively given by
\begin{align*}
    u \cdot_r v := u \cdot_2 v + \mu_{r(u)} v + \mu_{r(v)} u ~~~~ \text{ and } ~~~~ \{ u, v \}_r := \{ u, v\}_2 + \rho_{r(u)} v - \rho_{r(v)} u, \text{ for } u, v \in P_2.
\end{align*}
The Poisson algebra $(P_2, ~ \! \cdot_r ~ \! , \{ ~, ~ \}_r)$ has a representation on the space $P_1$ with the action maps $\nu_r , \psi_r : P_2 \rightarrow \mathrm{End}(P_1)$ given by
\begin{align*}
    (\nu_r)_u x = r(u) \cdot_1 x - r (\mu_x u) ~~~~ \text{ and } ~~~~ (\psi_r)_u x = \{ r(u), x \}_1 + r (\rho_x u), \text{ for } u \in P_2, x \in P_1.
\end{align*}
The corresponding FGV cohomology is defined to be the cohomology of the relative Rota-Baxter operator $r$ (of weight $1$).

\medskip

\noindent ({\bf Cohomology of Crossed homomorphisms.}) As before, let $(P_1, ~ \! \cdot_1 ~ \! , \{ ~, ~ \}_1)$ be a Poisson algebra that acts on another Poisson algebra $(P_2, ~ \! \cdot_2 ~ \! , \{ ~, ~ \}_2)$ by the maps $\mu, \rho : P_1 \rightarrow \mathrm{End} (P_2)$. Let $D : P_1 \rightarrow P_2$ be a crossed homomorphism on $P_1$ with values in $P_2$. Then the Poisson algebra $(P_1, ~ \! \cdot_1 ~ \! , \{ ~, ~ \}_1)$ has a new representation on the space $P_2$ with the action maps $\mu_D, \rho_D : P_1 \rightarrow \mathrm{End}(P_2)$ given by
\begin{align*}
    (\mu_D)_x u = \mu_x u + D(x) \cdot_2 u ~~~~ \text{ and } ~~~~  (\rho_D)_x u =\rho_x u + \{D (x), u \}_2, \text{ for } x \in P_1, u \in P_2
.\end{align*}
Then the FGV cohomology of the Poisson algebra $(P_1, ~ \! \cdot_1 ~ \! , \{ ~, ~ \}_1)$ with coefficients in the representation $(P_2, \mu_D, \rho_D)$ is defined to be the cohomology of the crossed homomorphism $D$.

\medskip

\noindent  ({\bf Cohomology of twisted Rota-Baxter operators.}) Let $(P, ~ \! \cdot ~ \! , \{ ~, ~ \})$ be a Poisson algebra, $(V, \mu, \rho)$ be a representation of it and $(h, H) \in Z^2 (P, V)$ be a Poisson $2$-cocycle. Let $r: V \rightarrow P$ be an $(h, H)$-twisted Rota-Baxter operator. Then the vector space $V$ can be given a Poisson algebra structure whose multiplication and bracket are respectively given by
\begin{align*}
    u \cdot_r v := \mu_{r(u)} v + \mu_{r (v)} u + h (r(u), r(v)) ~~~~  \text{ and } ~~~~ \{ u, v \}_r := \rho_{r(u)} v - \rho_{r(v) } u + H (r(u), r(v)), \text{ for } u, v \in V.
\end{align*}
The Poisson algebra $(V, ~ \! \cdot_r ~ \! , \{ ~, ~ \}_r)$ has a representation of the vector space $P$ with the action maps 
\begin{align*}
    (\nu_r)_u x = r(u) \cdot x - r (\mu_x u + h (r (u), x)) ~~~ \text{ and } ~~~ (\psi_r)_u x = \{ r (u) , x \} - r (- \rho_x u + H (r(u), x)),
\end{align*}
for $u \in V$, $x \in P$. The FGV cohomology of the Poisson algebra $(V, ~ \! \cdot_r ~ \! , \{ ~, ~ \}_r)$ with coefficients in the representation $(P, \nu_r, \psi_r)$ is defined to be cohomology of the operator $r$.

\medskip

\noindent ({\bf Cohomology of Reynolds operators.}) Let $(P, ~ \! \cdot ~ \! , \{ ~, ~ \})$ be a Poisson algebra and $r: P \rightarrow  P$ be a Reynolds operator on it. Then the space $P$ inherits a new Poisson algebra structure with the operations
\begin{align*}
    x \cdot_r y := r(x) \cdot y + x \cdot r(y)- r(x) \cdot r(y) ~~~~ \text{ and } ~~~~ \{ x, y \}_r := \{ r(x), y \}+\{x, r(y) \} - \{ r(x) , r(y) \},
\end{align*}
for $x, y \in P$. This Poisson structure has a representation on the vector space $P$ itself with the action maps
\begin{align*}
    (\mu_r)_x y = r(x) \cdot y - r ( x \cdot y - r(x) \cdot y) ~~~~ \text{ and } ~~~~ (\rho_r)_x y = \{ r(x) , y \} - r ( \{ x, y \} - \{ r(x), y \}), \text{ for } x, y \in P.
\end{align*}
The FGV cohomology of the Poisson algebra $(P, ~ \! \cdot_r ~ \!, \{ ~, ~ \}_r)$ with coefficients in the above representation on $P$ is said to be the cohomology of the Reynolds operator $r$.

\medskip

\noindent ({\bf Cohomology of modified Rota-Baxter operators.}) Let $(P, ~ \! \cdot ~ \! , \{ ~, ~ \})$ be a Poisson algebra and $r: P \rightarrow P$ be a modified Rota-Baxter operator on it. Then there is a new Poisson algebra structure on the vector space $P$ with the operations
\begin{align*}
    x \cdot_r y := r (x) \cdot y + x \cdot r (y) ~~~~ \text{ and } ~~~~ \{ x, y \}_r := \{ r(x), y\} + \{ x, r (y) \}, \text{ for } x, y \in P. 
\end{align*}
Moreover, this Poisson algebra $(P, ~ \! \cdot_r ~ \!, \{ ~, ~ \}_r)$ has a representation on the space $P$ itself with the action maps
\begin{align*}
    (\mu_r )_x y = r(x) \cdot y - r(x \cdot y) ~~ \text{ and } ~~ (\rho_r)_x y = \{ r(x), y\} - r \{ x, y \}, \text{ for } x, y  \in P.
\end{align*}
We call the cohomology of the Poisson algebra $(P, ~ \! \cdot_r ~ \!, \{ ~, ~ \}_r)$ with coefficients in the representation above as the cohomology of the modified Rota-Baxter operator $r$.

\medskip

\section{Linear and formal deformation theory of deformation maps}\label{sec7}
In this section, we study linear and formal deformation theory of a deformation map $r$ in a given proto-twilled Poisson algebra. In particular, we define Nijenhuis elements associated with a deformation map $r$ that generate linear deformations. Finally, we also find a sufficient condition for the rigidity of a deformation map $r$.

Let $(P_1 \oplus P_2, \odot, \{ \! \! \{ ~, ~ \} \! \! \})$ be a proto-twilled Poisson algebra and $r : P_2 \rightarrow P_1$ be a deformation map on it. A {\bf linear deformation} of $r$ is given by a parametrized sum $r_t = r + t ~ \!r_1$ (for some fixed $r_1 \in \mathrm{Hom}(P_2, P_1)$) that is also a deformation map for each value of the parameter $t \in {\bf k}$. In this case, we say that the map $r_1 \in \mathrm{Hom} (P_2, P_1)$ generates a linear deformation of $r$. It follows that $r_1 \in \mathrm{Hom}(P_2, P_1)$ generates a linear deformation of $r$ if and only if the following set of identities hold:
\begin{align}
    &r(u) \cdot_1 r_1 (v) + r_1 (u) \cdot_1 r(v) + \nu_u r_1 (v) + \nu_v r_1 (u)  \\
    & \qquad = r \big( \mu_{r_1 (u)} v + \mu_{r_1 (v)} u + h (r_1 (u),  r(v)) + h (r (u),  r_1 (v)) \big) + r_1 (u \cdot_r v ), \nonumber
    \end{align}
    \begin{align}
    r_1 (u) \cdot_1 r_1 (v) =  r_1 \big(  \mu_{r_1 (u)} v + \mu_{r_1(v)} u + h ( r(u) , r_1 (v) ) + h (r_1 (u), r(v))   \big) + r \big( h (r_1(u), r_1 (v)) \big),
     \end{align}
    \begin{align}
    r_1 \big(   h (r_1 (u) , r_1(v)) \big) = 0,
    \end{align}
    \begin{align}
    & \{ r(u), r_1 (v) \}_1 + \{ r_1 (u), r(v) \}_1 + \psi_u r_1 (v) - \psi_v r_1 (u)\\
    & \qquad = r \big( \rho_{r_1 (u)} v - \rho_{r_1 (v)} u + H (r_1 (u), r(v)) + H (r(u), r_1 (v))    \big) + r_1 (\{ u, v \}_r), \nonumber
     \end{align}
    \begin{align}
  \{ r_1 (u), r_1 (v) \}_1 = r_1 \big( \rho_{r_1 (u)} v - \rho_{r_1 (v)} u + H (r(u), r_1 (v)) + H (r_1 (u), r(v))    \big) + r \big(   H (r_1 (u), r_1 (v))  \big),
  \end{align}
    \begin{align}
 r_1 \big(   H (r_1 (u), r_1 (v))  \big) = 0,
\end{align}
for all $u, v \in P_2$.

Before we define `equivalence' between linear deformations, we need to first define homomorphisms between deformation maps. Let $( P_1 \oplus P_2, \odot, \{ \! \! \{ ~, ~ \} \! \! \})$ be a proto-twilled Poisson algebra and $r, r' : P_2 \rightarrow P_1$ be two deformation maps. A homomorphism from $r$ to $r'$ is a pair $(\varphi_1, \varphi_2)$ of linear maps $\varphi_1 : P_1 \rightarrow P_1$ and $\varphi_2 : P_2 \rightarrow P_2$ that make the map $\varphi_1 \oplus \varphi_2 : P_1 \oplus P_2 \rightarrow P_1 \oplus P_2$, $(x, u) \mapsto (\varphi_1 (x), \varphi_2 (u))$ into a homomorphism of Poisson algebras satisfying $\varphi_1 \circ r = r' \circ \varphi_2$.

\begin{definition}
    Let $r_t = r + t ~ \! r_1$ and $r_t' = r + t ~ \! r_1'$ be two linear deformations of $r$. Then $r_t$ is said to be {\em equivalent} with $r_t'$ if there exists an element $x_0 \in P_1$ such that the pair of maps
    \begin{align*}
        \big(  \varphi_{1,t} = \mathrm{Id}_{P_1} + t \{ x_0 , - \}_1 ~ \!, ~ \! \varphi_{2,t} = \mathrm{Id}_{P_2} + t ( \rho_{x_0} + H (x_0 , r - ))    \big)
    \end{align*}
    defines a homomorphism of deformation maps from $r_t$ to $r_t'$.
\end{definition}

Thus, for the above $\varphi_{1,t}$ and $\varphi_{2,t}$, we must have
\begin{align*}
    (\varphi_{1,t} \oplus \varphi_{2,t} ) ((x, u) \odot (y, v)) =~& (\varphi_{1,t} \oplus \varphi_{2,t})(x, u) \odot (\varphi_{1,t} \oplus \varphi_{2,t}) (y, v),\\
    (\varphi_{1,t} \oplus \varphi_{2,t} ) \{ \! \! \{ (x, u), (y, v) \} \! \! \} =~& \{ \! \! \{ (\varphi_{1,t} \oplus \varphi_{2,t} ) (x, u) , (\varphi_{1,t} \oplus \varphi_{2,t} ) (y, v) \} \! \! \},\\
    \varphi_{1,t} \circ r_t =~& r_t' \circ \varphi_{2,t}.
\end{align*}
Note that the last condition is equivalent to the following: for $u \in P_2$,
\begin{align}
    r_1 (u) - r_1' (u) =~& \{ r(u), x_0 \}_1 - r \big(  - \rho_{x_0} u + H (r(u), x_0) \big), \label{nij-ele} \\
    \{ r_1 (u) , x_0 \}_1 =~& r_1' \big(   -\rho_{x_0} u + H (r(u), x_0) \big). \label{nij-elem}
\end{align}
A linear deformation $r_t = r + t ~ \! r_1$ of a deformation map $r$ is said to be {\em trivial} if it is equivalent to the (undeformed) linear deformation $r_t' = r$. We will now consider Nijenhuis elements associated with a deformation map $r$ that generate trivial linear deformations.

\begin{definition}
Let $ (P_1 \oplus P_2, \odot, \{ \! \! \{ ~, ~ \} \! \! \})$ be a proto-twilled Poisson algebra and $r: P_2 \rightarrow P_1$ be a deformation map on it.
An element $x_0 \in P_1$ is said to be a {\em Nijenhuis element} associated with the deformation map $r$ if for all values of the parameter $t \in {\bf k}$, the map
    \begin{align*}
        P_1 \oplus P_2 \rightarrow P_1 \oplus P_2, ~ (x, u) \mapsto \big( x + t \{ x_0 , x \}_1 ~ \! , ~ \! u + t ( \rho_{x_0} u  + H (x_0, r (u)) ) \big)
    \end{align*}
    is a Poisson algebra homomorphism satisfying additionally
    \begin{align}\label{deform-nij}
        \big\{ \{ r(u), x_0 \}_1 - r (- \rho_{x_0} u + H (r(u), x_0 ) ) ~ \! , ~ \! x_0 \}_1 = 0, \text{ for all } u \in P_2.
    \end{align}
    The set of all Nijenhuis elements associated with $r$ is denoted by $\mathrm{Nij} (r)$.
\end{definition}

Let $r_t = r + t ~ \! r_1$ be a trivial linear deformation of $r$. Then there exists $x_0 \in P_1$ such that the pair of maps
$\big(   \varphi_{1,t} = \mathrm{Id}_{P_1} + t \{ x_0 , - \}_1 ~ \!, ~ \! \varphi_{2,t} = \mathrm{Id}_{P_2} + t ( \rho_{x_0} + H (x_0 , r - ) )  \big)$ defines a homomorphism of deformation maps from $r_t = r + t ~ \! r_1$ to $r_t' = r$. In the identities (\ref{nij-ele}) and (\ref{nij-elem}), by substituting $r_1' = 0$, we obtain (\ref{deform-nij}). Hence, it turns out that $x_0 \in P_1$ is a Nijenhuis element associated with the deformation map $r$. Conversely, if the map $\psi= 0$ in the proto-twilled Poisson algebra (see (\ref{12-maps}))  and $x_0 \in \mathrm{Nij}(r)$ is a Nijenhuis element associated with $r$, then it is not hard to see that $r_t = r + t ~ \! \delta^r_\mathrm{FGV} (x_0)$ is a trivial linear deformation of $r$.

\medskip

Let $(\mathcal{P} = P_1 \oplus P_2, \odot, \{ \! \! \{ ~, ~ \} \! \! \})$ be a proto-twilled Poisson algebra and $r : P_2 \rightarrow P_1$ be a deformation map on it. A {\bf formal one-parameter deformation} of $r$ is given by a formal power series
\begin{align*}
    r_t = r_0 + t r_1 + t^2 r_2 + \cdots \in \mathrm{Hom}(P_2, P_1) [ \! \! [ t] \! \! ] \text{ with } r_0 = r
\end{align*}
such that its ${\bf k} [ \! \! [t ] \! \! ]$-linear extension map (also denoted by the same notation) $r_t : P_2 [ \! \! [ t] \! \! ] \rightarrow P_1 [ \! \! [ t] \! \! ]$ is a deformation map in the proto-twilled Poisson algebra $(\mathcal{P} [ \! \! [ t] \! \! ] \cong P_1 [ \! \! [ t] \! \! ] \oplus P_2 [ \! \! [ t] \! \! ], \odot, \{ \! \! \{ ~, ~ \} \! \! \}).$ Here the Poisson algebra structure of $\mathcal{P}$ has been extended to a Poisson algebra structure on $\mathcal{P} [ \! \! [ t ] \! \! ]$ by  using ${\bf k} [ \! \! [ t ] \! \! ]$-(bi)linearity.

It turns out that $r_t = \sum_{i=0}^\infty t^i r_i $ (with $r_0 = r$) is a formal one-parameter deformation of $r$ if and only if for each $n \geq 1$ and $u, v \in P_2$,
\begin{align*}
    \sum_{i+j = n} r_i (u) \cdot_1 r_j (v) \! ~&+~ \! \nu_u r_n (v) + \nu_v r_n (u) \\
    =~& r_n (u \cdot_2 v) + \sum_{i+j = n} r_i \big(  \mu_{r_j (u)} v +\mu_{r_j (v)} u  \big) + \sum_{i+j+k = n} r_i \big(  h (r_j (u), r_k (v))   \big),\\
    \sum_{i+j = n} \{ r_i (u), r_j (v) \}_1  \! ~ &+ ~ \! \psi_u r_n (v) - \psi_v r_n (u) \\
    =~& r_n (\{ u, v \}_2) + \sum_{i+j = n} r_i \big(   \rho_{r_j (u)} v - \rho_{r_j (v)} u  \big) + \sum_{i+j + k = n} r_i \big(  H (r_j (u), r_k (v))  \big).
\end{align*}
For $n =1$, the above two identities are respectively equivalent to $\delta^{1,0}_\mathrm{H} (r_1) (u, v) = 0$ (or $\delta^{0,1}_\mathrm{H} (r_1) (u, v) = 0$) and $\delta^{0,1}_\mathrm{CE} (r_1) (u, v) = 0$. Hence
\begin{align*}
    \delta^r_\mathrm{FGV} (r_1) = \delta_\mathrm{H}^{0,1} (r_1) + \delta_\mathrm{CE}^{1,0} (r_1) = 0.
\end{align*}
This shows that the map $r_1 \in \mathrm{Hom} (P_2, P_1) = C^{0,1} \big( (P_2)_r, P_1 \big) = C^1 ((P_2)_r, P_1)$ is a $1$-cocycle in the cochain complex associated with the deformation map $r$. This is called the {\em infinitesimal} of the given formal deformation $r_t = \sum_{i=0}^\infty r_i t^i$. 

\begin{definition}
    Let $r_t = \sum_{i=0}^\infty t^i r_i $ and $r_t' =  \sum_{i=0}^\infty t^i r_i'$ (with $r_0 = r_0' = r$) be two formal one-parameter deformations of $r$. They are said to be {\bf equivalent} if there exists $x_0 \in P_1$ and maps $\varphi_{1,i} \in \mathrm{Hom} (P_1, P_1)$, $\varphi_{2,i} \in \mathrm{Hom} (P_2, P_2)$ for $i \geq 2$ such that the pair of maps
    \begin{align*}
        \big( \varphi_{1,t} = \mathrm{id}_{P_1} + t \{ x_0 , - \}_1 + \sum_{i \geq 2} t^i ~\!  \varphi_{1,i}  ~ \!  , ~ \! \varphi_{2,t} =  \mathrm{id}_{P_2} + t ( \rho_{x_0} + H (x_0, r -)) + \sum_{i \geq 2} t^i ~\! \varphi_{2,i} \big)
    \end{align*}
    defines a homomorphism of deformation maps from $r_t$ to $r_t'$.
\end{definition}
By expanding the condition $\varphi_{1,t} \circ r_t = r_t' \circ \varphi_{2,t}$ and equating coefficients of $t$ in both sides, we simply obtain the identity (\ref{nij-ele}). This, in turn, implies that (when the map $\psi=0$)
\begin{align*}
    r_1 (u) - r'_1 (u) = (\psi_r)_u x_0 = \delta^r_\mathrm{FGV} (x_0) (u), \text{ for } u \in P_2.
\end{align*}
As a summary, we get the following.

\begin{theorem}
    Let $(\mathcal{P} = P_1 \oplus P_2, \odot, \{ \! \! \{ ~, ~ \} \! \! \})$ be a proto-twilled Poisson algebra and $r : P_2 \rightarrow P_1$ be a deformation map on it. Then the infinitesimal of any formal one-parameter deformation of $r$ is a $1$-cocycle in the cochain complex associated with $r$. Moreover, when $\psi = 0$, the infinitesimals of equivalent formal one-parameter deformations are cohomologous, i.e., they correspond to the same element in $H^1_\mathrm{FGV} ( (P_2)_r, P_1)$.
\end{theorem}

In the following, we shall consider the rigidity of a deformation map $r$ and find a sufficient condition for the rigidity.

\begin{definition}
    Let $(\mathcal{P} = P_1 \oplus P_2, \odot, \{ \! \! \{ ~, ~ \} \! \! \})$ be a proto-twilled Poisson algebra. A deformation map $r: P_2 \rightarrow P_1$ is said to be {\bf rigid} if any formal one-parameter deformation $r_t = \sum_{i=0}^\infty t^i r_i $ of $r$ is equivalent to the undeformed one $r_t' = r$.
\end{definition}

\begin{theorem}
Let $(\mathcal{P} = P_1 \oplus P_2, \odot, \{ \! \! \{ ~, ~ \} \! \! \})$ be a proto-twilled Poisson algebra in which $\psi = 0$. Let $r: P_2 \rightarrow P_1$ is a deformation map such that $Z^1 ( (P_2)_r, P_1) = \delta_r ( \mathrm{Nij} (r))$, where $Z^1 ( (P_2)_r, P_1)$ is the set of all $1$-cocycles in the cochain complex associated with $r$. Then $r$ is rigid.
\end{theorem}

\begin{proof}
    Take any formal one-parameter deformation $r_t = \sum_{i=0}^\infty t^i r_i $ of the operator $r$. Then we have seen that $r_1$ is a $1$-cocycle in the cochain complex associated with $r$. Hence from the hypothesis, we get that $r_1 = \delta^r_\mathrm{FGV} (x_0)$, for some $x_0 \in \mathrm{Nij} (r)$. We set
    \begin{align*}
        \varphi_{1,t} = \mathrm{Id}_{P_1} + t \{ x_0 , - \}_1 ~~~~  \text{ and } ~~~~ \varphi_{2,t} = \mathrm{Id}_{P_2} + t \big(  \rho_{x_0} + H (x_0 , r -)   \big).
    \end{align*}
    Then $r_t' := \varphi_{1,t} \circ r_t \circ \varphi_{2,t}^{-1}$ is a formal one-parameter deformation of $r$ and this is equivalent to $r_t$. We observe that
    \begin{align*}
        r_t' (u) \quad (\mathrm{mod }~ t^2) =~& (\mathrm{Id}_{P_1} + t \{ x_0, - \}_1) (r+ t r_1) (\mathrm{Id}_{P_2} - t (\rho_{x_0} + H (x_0, r -)))(u)  \quad (\mathrm{mod }~ t^2)\\
        =~&  (\mathrm{Id}_{P_1} + t \{ x_0, - \}_1) \big( r(u) + t ~ \! r_1 (u)- t ~ \! r (\rho_{x_0} u + H (x_0 , r(u))) \big)  \quad (\mathrm{mod }~ t^2) \\
        =~& r(u) + t \big(   r_1 (u) + \{ x_0, r(u) \}_1 - r (\rho_{x_0} u + H (x_0, r(u)) )   \big) \quad  (\mathrm{mod }~ t^2)  \\
        =~& r (u) \qquad (\because ~  r_1 = \delta^r_\mathrm{FGV} (x_0) ).
    \end{align*}
    This shows that the coefficient of $t$ in the power series expansion of $r_t'$ vanishes. In the same way, we can show that $r_t'$ is equivalent to some formal one-parameter deformation $r_t''$ whose coefficients in $t$ and $t^2$ both vanish. Applying this process repeatedly, we get that $r_t$ is equivalent to the undeformed one $r_t' = r$. This proves the result.
\end{proof}


\begin{remark}
    In the last part of this paper, we developed the cohomologies of various well-known operators on Poisson algebras. In particular, we defined the cohomology of a (relative) Rota-Baxter operator of weight $0$. On the other hand, it is well-known \cite{aguiar} that a (relative) Rota-Baxter operator of weight $0$ on a Poisson algebra induces a pre-Poisson algebra structure. Recall that a pre-Poisson algebra is a vector space endowed with a zinbiel algebra (or a commutative dendriform algebra) structure and a pre-Lie algebra structure satisfying certain compatibility conditions. It would be very interesting to find the cohomology theory of a pre-Poisson algebra (combining the cohomologies of the underlying zinbiel algebra and pre-Lie algebra) that controls the formal one-parameter deformations of the structure. Further, one may look for a homomorphism from the cohomology of a (relative) Rota-Baxter operator of weight $0$ on a Poisson algebra to the cohomology of the induced pre-Poisson algebra.
\end{remark}

\medskip

    \noindent {\bf Acknowledgements.} R. Mandal and A. Sahoo would like to acknowledge the financial support received from the Government of India through PMRF fellowships. 
    All the authors thank the Department of Mathematics, IIT Kharagpur for providing the beautiful academic atmosphere where the research has been done.
    
\medskip

\noindent {\bf Conflict of interest statement.} There is no conflict of interest.

\medskip

\noindent {\bf Data Availability Statement.} Data sharing does not apply to this article as no new data were created or analyzed in this study.

\end{document}